\documentclass[10pt]{article}
\usepackage{a4}

\usepackage{amsmath, amsfonts,amssymb,amscd,amsthm}
\usepackage{graphicx}

\usepackage{bm}
\usepackage{subfig}
\usepackage{graphicx}
\usepackage{pict2e}
\usepackage{color}
\usepackage[percent]{overpic}
\usepackage[colorlinks=true,linkcolor=black,urlcolor=black]{hyperref}
\usepackage[english]{babel}
\usepackage{algorithm, algorithmic}

\usepackage[title]{appendix}

\newtheorem{theorem}{Theorem}
\newtheorem{theorem*}{Theorem}

\newtheorem{proposition}[theorem]{Proposition}
\newtheorem{definition}[theorem]{Definition}
\newtheorem{lemma}[theorem]{Lemma}

\newtheorem{remark}[theorem]{Remark}

\DeclareMathOperator{\R}{\mathbb{R}}

\begin{document}
\title{Time domain boundary elements for elastodynamic contact}
\author{Alessandra Aimi\thanks{Department of Mathematical, Physical and Computer Sciences,
University of Parma, Parco Area delle Scienze, 53/A, 43124, Parma,
Italy, email: alessandra.aimi@unipr.it, giulia.dicredico@unipr.it \newline Members of the INDAM-GNCS Research Group, Italy.} \and Giulia Di Credico${}^\ast$\thanks{Engineering Mathematics, University of Innsbruck, Innsbruck, Austria,  email: heiko.gimperlein@uibk.ac.at} \and Heiko Gimperlein${}^\dagger$}
\date{}

\providecommand{\keywords}[1]{{\textit{Key words:}} #1}

\maketitle \vskip 0.5cm
\begin{abstract}
\noindent This article proposes a boundary element method for the dynamic contact between a linearly elastic body and a rigid obstacle. The Signorini contact problem is formulated as a variational inequality for the Poincar\'{e}-Steklov operator for the elastodynamic equations on the boundary, which is solved in a mixed formulation using boundary elements in the time domain. We obtain an a priori estimate for the resulting Galerkin approximations. Numerical experiments confirm the stability and convergence of the proposed method for the contact problem in flat and curved two-dimensional geometries, as well as for moving obstacles. 

\end{abstract}
\keywords{boundary element methods; space-time methods; Signorini contact problem; variational inequality; elastodynamics.}

\section{Introduction}
\label{intro}

Contact problems between elastic bodies play a crucial role in applications {ranging} from fracture dynamics to rolling car tires \cite{wrig}. As the contact takes place at the interface of two materials, boundary elements as well as the coupling of finite and boundary elements lead to efficient and much-studied tools for numerical simulations for time-independent situations as well as for strongly dissipative materials \cite{gwinsteph, ency}. Mathematically, the analysis of the resulting variational inequalities has been well understood for elliptic and parabolic problems.\\
In spite of the practical relevance of dynamic contact and the challenges of simulations based on finite elements, rigorous boundary element methods are only starting to be developed. We refer to  \cite{bur2,chouly0, choulydyn2,doyen2,fractional, sten2,hauret,cont2,cont1,chouly3} as examples from the extensive mathematical finite element literature. Indicating the challenges, without dissipation only stability or energy conservation, but not the convergence of finite element methods, is known.\\
The computational difficulties relate to the mathematical challenges of the time-dependent contact problem \cite{eckbook}. Even the existence of weak solutions is known only for viscoelastic materials or modified contact conditions, such as in \cite{cocou}. Recent progress towards the numerical analysis problems with unilateral constraints and without dissipative terms  was made in \cite{contact}, for the simplified problem of a scalar wave equation which arises in the physical limit when transversal stresses can be neglected \cite{lure}. Unlike for elastic contact, in this scalar problem refined information about the Poincar\'{e}-Steklov operator is available for the analysis \cite{cooper, lebeau}. \\
Building on \cite{contact} and \cite{banz}, in this work we propose a time domain boundary element method for the elastodynamic Signorini contact problem corresponding to contact with an impenetrable obstacle.\\

To be specific, we consider the dynamics of a linear elastic body with homogeneous mass density $\varrho$ in a {bounded} domain $\Omega \subset \mathbb{R}^d$, $d=2,3$. Its dynamics is  described by the Navier-Lam\'{e} equations
\begin{equation}\label{navierlame}
\nabla \cdot \sigma(\textbf{u})-\varrho\ddot{\textbf{u}}=\textbf{0} 
\end{equation}
for times $t \in [0,T]$. Here $\textbf{u}$ is the unknown displacement vector and the upper dots denote time derivatives. The Cauchy stress tensor $\sigma(\textbf{u})$ is defined as follows {(see e.g. \cite{Eringen1975})}
$$\sigma(\textbf{u})=\varrho (c_{\mathtt{P}}^2-2c_{\mathtt{S}}^2)\,(\nabla \cdot \textbf{u})I+\varrho c_{\mathtt{S}}^2\,(\nabla \textbf{u}+\nabla \textbf{u}^\top)\,,$$
with pressure and shear velocities $c_{\mathtt{P}}$, respectively $c_{\mathtt{S}}$. From $\sigma(\textbf{u})$ we obtain the elastic traction \begin{equation}
\label{trazione}
\textbf{p}=\sigma\left(\textbf{u}\right)_{\vert_{\Gamma}} \bf n\,\end{equation}
on the boundary $\Gamma$ of $\Omega$. Here $\textbf{n}$ denotes the outward-pointing unit normal vector to $\Gamma$. 

Indicating with the subscript $\perp$ the normal component of a vector, we consider the non-penetration of an obstacle using the following (non-linear) Signorini boundary conditions for 
${ u}_\perp, p_\perp$ on a given contact boundary $\Gamma_C \subset \Gamma$:
\begin{equation}\label{contactbc}\begin{cases}
{ u}_\perp \geq g\ ,\, {p}_\perp\geq {f}_\perp\ ,\\ {u}_\perp > g\ \Longrightarrow {p}_\perp = {f}_\perp\ .
\end{cases}\end{equation}
The mechanical interpretation of the gap function $g$ and the prescribed force ${\bf f}$ are discussed in Section \ref{mechsetup} below. As described there, the Signorini conditions are complemented by boundary conditions prescribing the displacement, respectively the traction, on the remainder of $\Gamma$. \\
In this article we reduce the contact problem to a variational inequality on the boundary $\Gamma$, using the elastodynamic Poincar\'{e}-Steklov operator. 
 Because the contact area and forces are often relevant in applications, we reformulate this variational inequality in an equivalent mixed formulation, which we discretize with a time domain Galerkin boundary element method. The mixed formulation is solved with an Uzawa
algorithm to obtain the displacement and the contact forces.\\
The detailed algebraic formulation and implementation of an energetic space-time boundary element method are presented in this article. As a main theoretical result we obtain an a priori estimate for the numerical error in Theorem \ref{apriori}, as well as the convergence of the proposed Uzawa algorithm. Given the analytical challenges described above, the theoretical results are subject to an assumption known for simplified situations \cite{contact}.\\
Numerical results confirm the stability and convergence of the proposed method in two dimensions. Both flat and curved contact boundaries are considered, as is the contact with time-dependent, moving obstacles.\\
Let us finally note that our approach relies on the recent advances in time domain boundary elements and coupled finite element / boundary element procedures for interface problems, where both space-time Galerkin and convolution quadrature methods have been of interest \cite{ajrt,aimi1,comput,falletta,gimperleinreview,jr,kager,sayas,schanz,steinbach}.
 More specifically, we refer
to \cite{AimiJCAM, ourpaper, Becache1993, Becache1994} for  works on the theoretical and numerical analysis of the elastodynamic boundary integral operators.\\
This article is structured as follows: in Section \ref{sec2} the differential model problem is set up together with its reformulation in terms of space-time boundary integral operators; in Section \ref{sec 3} the adopted discretization is introduced, while in Section \ref{sec4} the theoretical analysis is conducted. Algorithmic details of the overall implementation are described in Section \ref{sec:psh} and extensive numerical simulations are presented and discussed in Section \ref{sec;numres}. Conclusions are briefly drawn in Section \ref{sec7}.

\section{Dynamic contact and boundary integral formulations} \label{sec2}
\subsection{Differential problem formulation} \label{mechsetup}

\begin{figure}[h!]
\centering
{\includegraphics[scale=0.7]{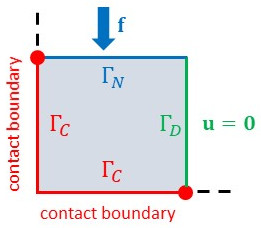}}
\caption{Schematic depiction of contact problem.}
\label{Figure:setup}
\end{figure}

In this subsection we introduce the equations which govern the contact problems considered in this article. As illustrated in 
Figure \ref{Figure:setup} for $d=2$, we consider the time dependent deformation of a linearly elastic body, described by the 
Navier-Lam\'{e} equations \eqref{navierlame} for the displacement $\textbf{u}$ in a {bounded} domain $\Omega \subset \mathbb{R}^d$, $d=2,3$. 
Starting from the reference configuration, $\textbf{u} = \textbf{0}$ for times $t\leq 0$, the dynamics of the body is due to surface 
forces applied on a subset $\Gamma_N$ of its boundary $\Gamma$, where the traction $\textbf{p} =\mathbf{f}$ is prescribed.   

The unilateral Signorini contact problem here considered describes the physical impossibility of the body to penetrate an adjacent
rigid and frictionless surface. In Figure \ref{Figure:setup} the impenetrable obstacle is given by the positive $x$- and $y$-axes.
Nonpenetration of the obstacle leads to the contact condition \eqref{contactbc} on the part $\Gamma_C \subset \Gamma$ of the boundary where contact may occur. 
More precisely, from \eqref{contactbc} contact takes place when the normal displacement satisfies $u_\perp = g$, where $g$ describes the gap between the reference configuration and the obstacle.  
Contact is avoided when $u_\perp > g$, and then only the applied surface forces ${p}_\perp = {f}_\perp$ act on $\Gamma_C$.

It will be convenient to denote by $\bar{\Gamma}_\Sigma$ the union $\bar{\Gamma}_N \cup \bar{\Gamma}_C$ of the traction and contact boundaries. To complete the description of the problem, the body is fixed on the remaining subset $\Gamma_D = \Gamma \setminus \bar{\Gamma}_\Sigma$ of the boundary. This results in the boundary condition $\textbf{u} = \textbf{0}$ for the displacement on $\Gamma_D$. \\
Indicating with the subscript $\parallel$ the component of a vector tangential to $\Gamma$, we can summarize the governing equations as follows: 
\begin{subequations}  \label{prob:strong_formulation}
\begin{alignat}{2}
\nabla \cdot \sigma(\textbf{u})-\varrho\ddot{\textbf{u}}&=\textbf{0} & \quad & \text{in } [0,T]\times \Omega \label{strongformPDE}\\
	\textbf{u}&=\textbf{0} & \quad & \text{on } [0,T]\times \Gamma_D \\
  \textbf{p} &=\mathbf{f} & \quad & \text{on } [0,T]\times \Gamma_N \label{prob:strong_formulation_1}\\
   \textbf{p}_\parallel &=\textbf{f}_\parallel  & \quad & \text{on } [0,T]\times \Gamma_C 
   \label{prob:strong_formulation_2}
\end{alignat}
\end{subequations}
together with the contact conditions \eqref{contactbc} on $[0,T]\times \Gamma_C$ and homogeneous 
condition $\mathbf{u}\equiv \mathbf{0}$  
in $\Omega$ for $t\leq 0$. \\

\subsection{Boundary integral formulations: variational inequality and mixed formulations}
\label{sec:Boundaryintegralformulation}

Boundary integral formulations are well-known to lead to efficient numerical methods for time dependent contact problems (see \cite{contact} in the context of acoustic wave equation), as the dynamics in $\Omega$ can be fully described by an integral equation on $[0,T]\times\Gamma$.\\
For the remainder of this article we assume that $\Omega$ is a bounded polygonal or polyhedral Lipschitz domain. To derive the corresponding formulations for the elastodynamic equations \eqref{navierlame}, we introduce the time-dependent single layer and double layer potential operators $V$ and $K$, that act onto the generic fields $\pmb{ \phi}$ and $\pmb{\psi}$, for $(t,{\bf x}) \in [0,T] \times \Omega$, as follows:
$$[V \pmb{\phi}](\bold{x},t) = \int_0^t\int_\Gamma G(t,\tau;\bold{x},\bold{y})\: {\pmb{\phi}}(\tau,\bold{y}) d\Gamma_{\bold{y}} d\tau,\:\:
[K \pmb{\psi}](\bold{ x},t) =\int_0^t \int_\Gamma  \left[\sigma_{\bold{y}}\left(G\right)^\top (t,\tau;\bold{x},\bold{y})\bold{n}_{\bold{y}}\right] \pmb{ \psi}(\bold{y},\tau)  d\Gamma_{\bold{y}} d\tau,
$$
where $G$ is the fundamental solution to \eqref{navierlame}.
The subscript $\bold{y}$ applied to the stress tensor $\sigma$ denotes the variable for the application of the spatial derivative, while for the vector $\bold{n}$ it declares the point of $\Gamma$ where we are considering the normal direction. For $d=2$ the fundamental solution reads as
\begin{align}
G_{ij}(\textbf{x},\textbf{y};t,\tau):=& \dfrac{H[c_{\mathtt{P}}(t-\tau)-r]}{2\pi\varrho c_{\mathtt{P}}}\left\lbrace \dfrac{r_i r_j}{r^4}\dfrac{2 c^2_{\mathtt{P}}(t-\tau)^2-r^2}{\sqrt{c_{\mathtt{P}}^2(t-\tau)^2-r^2}}-\dfrac{\delta_{ij}}{r^2}\sqrt{c^2_{\mathtt{P}}(t-\tau)^2-r^2}\right\rbrace \label{fundamental solution}\\
-& \dfrac{H[c_{\mathtt{S}}(t-\tau)-r]}{2\pi\varrho c_{\mathtt{S}}}\left\lbrace \dfrac{r_i r_j}{r^4}\dfrac{2 c^2_{\mathtt{S}}(t-\tau)^2-r^2}{\sqrt{c_{\mathtt{S}}^2(t-\tau)^2-r^2}}-\dfrac{\delta_{ij}}{r^2}\dfrac{c^2_{\mathtt{S}}(t-\tau)^2}{\sqrt{c^2_{\mathtt{S}}(t-\tau)^2-r^2}}\right\rbrace,\quad i,j=1,2,\nonumber
\end{align}
being $H[\cdot]$ the Heaviside step function, $r_j$ and $r$ the $j$-th cartesian component and the euclidean norm of the vector $\textbf{r}=\textbf{x}-\textbf{y}$, respectively, and $\delta_{ij}$ the Kronecker delta. For $d=3$ the fundamental solution reads instead as
\begin{align*}
G_{ij}(\textbf{x},\textbf{y};t,\tau):=&  \frac{t-\tau}{4\pi\varrho r^2} \left(\frac{r_i r_j}{r^3} - \frac{\delta_{ij}}{r}\right)(H[c_{\mathtt{P}}(t-\tau)-r]-H[c_{\mathtt{S}}(t-\tau)-r]) \\ &+ \frac{r_i r_j}{4\pi\varrho r^{3}} \left(c_{\mathtt{P}}^{-2}\delta(c_{\mathtt{P}}(t-\tau)-r)-c_{\mathtt{S}}^{-2}\delta(c_{\mathtt{S}}(t-\tau)-r) \right) \\& + \frac{\delta_{ij}}{4\pi\varrho r c_{\mathtt{S}}^2} \delta(c_{\mathtt{S}}(t-\tau)-r), \qquad i,j=1,2,3,
\end{align*}
where $\delta (\cdot)$ is the Dirac delta distribution.\\

The unknown displacement $\textbf{u}$ can be expressed  by the representation formula 
\begin{equation}\label{representation formula}
\textbf{u}=V\textbf{p}-K\textbf{u},\quad \textrm{in} \; [0,T] \times \Omega.
\end{equation}
Letting $\textbf{x}\in\Omega\rightarrow \textbf{x}\in \Gamma$ in \eqref{representation formula}, we deduce the boundary integral equation in the standard notation
\begin{equation}\label{first_BIE}
\frac{1}{2}\textbf{u}=\mathcal{V}\textbf{p}-\mathcal{K}\textbf{u},\quad \textrm{in}\; [0,T] \times \Gamma,
\end{equation}
taking into account the free term $\frac{1}{2}\textbf{u}$ generated by the operator $K$ pursuant the limiting process.
Moreover, we need also to introduce the adjoint double layer operator $\mathcal{K}^\star$ and the hypersingular integral operator $\mathcal{W}$ for $(\textbf{x},t)\in \Gamma\times [0,T]$:
\begin{align*}
[\mathcal{K}^\star \pmb{\phi}](\bold{ x},t)& = \int_0^t\int_\Gamma \left[\sigma_{\bold{x}} \left(G\right)(t,\tau;\bold{x},\bold{y})\bold{n}_{\bold{x}}\right]  {\pmb{\phi}}(\tau,\bold{y}) d\Gamma_{\bold{y}}\ d\tau,\\
[\mathcal{W} \pmb{\psi}](\bold{ x},t) &=\int_0^t \int_\Gamma  \left[\sigma_\textbf{x}\left(\sigma_{\bold{y}}\left(G\right)^\top (t,\tau;\bold{x},\bold{y})\bold{n}_{\bold{y}}\right)\textbf{n}_{\textbf{x}}\right] \pmb{ \psi}(\bold{y},\tau)  d\Gamma_{\bold{y}}\ d\tau\,,
\end{align*}
involved in the classical boundary integral equation
\begin{equation}\label{second_BIE}
\frac{1}{2}\textbf{p}=\mathcal{K}^*\textbf{p}-\mathcal{W}\textbf{u},\quad \textrm{in}\; [0,T] \times \Gamma.
\end{equation}
The above recalled integral operators will be used, as detailed in Section \ref{subPS}, in the boundary integral problem we are going to solve, which is based on the Poincar\'e-Steklov operator $\mathcal{S}$, defined by
\begin{equation}\label{operator_S}
\mathcal{S} \left({\bf u}_{|_\Gamma}\right) := \sigma\left(\textbf{u}\right)_{|_\Gamma}  {\bf n}=\mathbf{p},
\end{equation}
where  ${\bf u}$ is a solution to the elastodynamic equations \eqref{navierlame} in $[0,T]\times\Omega$, given a Dirichlet datum ${\bf u}_{|_\Gamma}$, and the last equality is due to the definition of traction in \eqref{trazione}.\\ 
Basic properties of $\mathcal{S}$ are summarized in Theorem \ref{mappingProperties} below. Moreover, let us remark that this operator
can be expressed by two equivalent forms, matching properly the previously introduced integral operators, namely: the non-symmetric formulation
\begin{equation}\label{definition_S}
\mathcal{S}=\mathcal{V}^{-1}\left(\mathcal{K}+\frac{1}{2}\right)
\end{equation}
and the symmetric formulation
\begin{equation}\label{definition_S_symm}
\mathcal{S}=\left(\mathcal{K}^*+\frac{1}{2}\right)\mathcal{V}^{-1}\left(\mathcal{K}+\frac{1}{2}\right)-\mathcal{W} \,.
\end{equation}

In the following, to simplify the notation the subscript $\vert_\Gamma$ in the argument of the operator {$\cal S$} will be omitted, whenever clear from the context.\\

Now, let us denote $L^2$ space-time scalar products useful in the sequel by
\begin{align}
&\langle \textbf{u},\textbf{v} \rangle_{0,\Gamma,(0,T]}:=\int_0^T \int_{\Gamma}\textbf{u}(\textbf{x},t)\cdot\textbf{v}(\textbf{x},t)\: d\Gamma_{\textbf{x}}\:dt,\label{0_product}\\
&\langle \textbf{u},\textbf{v} \rangle_{\sigma,\Gamma,\mathbb{R}^+}:=\int_0^\infty e^{-2\sigma t}\int_{\Gamma}\textbf{u}(\textbf{x},t)\cdot\textbf{v}(\textbf{x},t)\: d\Gamma_{\textbf{x}}\:dt,\label{sigma_product}  
\end{align}
where $\sigma>0$, changing the subscript $\Gamma$ whenever dealing with boundary subsets and including in the above definitions the simpler case of scalar functions.
For precise statements and the error analysis we also require space-time Sobolev spaces $H^r(\mathcal{I},\tilde{H}^{s}({\Gamma'}))$ on subsets ${\Gamma'} \subset \Gamma$ and for the time intervals $\mathcal{I}=[0,T],\mathbb{R}^+$, which are  introduced in Appendix A.\\

Given $g \in H^{1/2}([0,T],H^{1/2}(\Gamma_C))$ and $\textbf{f}\in H^{1/2}([0,T],\tilde{H}^{-1/2}(\Gamma_\Sigma))^d$, the precise functional analytic formulation of the contact problem \eqref{contactbc}-\eqref{prob:strong_formulation} as a variational inequality in terms of $\mathcal{S}$ reads:\\

\noindent \textit{find} $\textbf{u} \in \mathcal{C}:=\left\lbrace \textbf{v}\in H^{1/2}([0,T],\tilde{H}^{1/2}(\textcolor{black}{\Gamma_\Sigma}))^d: v_\perp \geq g \text{~a.e.~on~} [0,T]\times \Gamma_C \right\rbrace$ \textit{such that}
\begin{align} \label{eq:VarIneq}
\langle {\mathcal{S} {\bf u}}, {\bf v}-{\bf u}\rangle_{0,\Gamma_\Sigma,(0,T]} \geq \langle {\bf f}, {\bf v}-{\bf u}\rangle_{0,\Gamma_\Sigma,(0,T]} \qquad \quad  \forall {\bf v}  \in \mathcal{C}.
\end{align}

For the following result, we recall that,
as typical for direct boundary element methods, the displacement $\mathbf{u}$ in the domain $[0,T]\times\Omega$ is recovered from its boundary trace $\mathbf{u}|_\Gamma$ on $[0,T]\times\Gamma$ using the representation formula \eqref{representation formula}.

\begin{proposition}
The variational inequality \eqref{eq:VarIneq} for the displacement $\mathbf{u}|_\Gamma$ on $[0,T]\times\Gamma$ is equivalent to the contact problem \eqref{contactbc}-\eqref{prob:strong_formulation} for the solution $\mathbf{u}$ in $[0,T]\times\Omega$.
\end{proposition}

\begin{proof}
First, we show that the boundary trace $\mathbf{u}|_\Gamma$ of the solution $\mathbf{u}$ to \eqref{contactbc}-\eqref{prob:strong_formulation} satisfies \eqref{eq:VarIneq}.  Note that the Signorini boundary condition \eqref{contactbc} on $\Gamma_C$ can be  expressed in terms of $\mathcal{S}$ using \eqref{operator_S}:
\begin{equation}\label{scontactbc}\begin{cases}
{ u}_\perp \geq g\ ,\, \mathcal{S}(\mathbf{u}|_\Gamma)_\perp\geq {f}_\perp\ ,\\ {u}_\perp > g\ \Longrightarrow \mathcal{S}(\mathbf{u}|_\Gamma)_\perp = {f}_\perp\ .
\end{cases}\end{equation}
Equivalently, it is written as
\begin{equation}\label{scontactbc2}
{ u}_\perp \geq g\ ,\ \mathcal{S}(\mathbf{u}|_\Gamma)_\perp\geq {f}_\perp\ , \ ({u}_\perp - g)(\mathcal{S}(\mathbf{u}|_\Gamma)_\perp - f_\perp) = 0\ .\end{equation}
Due to \eqref{prob:strong_formulation_2}, the parallel component on $\Gamma_C$ satisfies $$\mathcal{S}(\mathbf{u}|_\Gamma)_\parallel = f_\parallel.$$
Further, from \eqref{prob:strong_formulation_1}, the Neumann boundary conditions on $\Gamma_N$ are equivalent to
$$\mathcal{S}(\mathbf{u}|_\Gamma) = \mathbf{f}.$$
We therefore find for any $\mathbf{v} \in \mathcal{C}$
\begin{align*}\langle\mathcal{S}(\mathbf{u}|_\Gamma) - \mathbf{f},\mathbf{v}-\mathbf{u}|_\Gamma\rangle_{0,\Gamma_\Sigma,(0,T]} &= \langle\mathcal{S}(\mathbf{u}|_\Gamma)_\parallel - f_\parallel,(\mathbf{v}-\mathbf{u}|_\Gamma)_\parallel\rangle_{0,\Gamma_\Sigma,(0,T]} + \langle\mathcal{S}(\mathbf{u}|_\Gamma)_\perp - f_\perp,(\mathbf{v}-\mathbf{u}|_\Gamma)_\perp\rangle_{0,\Gamma_\Sigma,(0,T]}\\
& = 0 + \langle\mathcal{S}(\mathbf{u}|_\Gamma)_\perp - f_\perp,v_\perp-g\rangle_{0,\Gamma_\Sigma,(0,T]}-\langle\mathcal{S}(\mathbf{u}|_\Gamma)_\perp - {f}_\perp,u_\perp-g\rangle_{0,\Gamma_\Sigma,(0,T]}\\
& = \langle\mathcal{S}(\mathbf{u}|_\Gamma)_\perp - f_\perp,v_\perp-g\rangle_{0,\Gamma_\Sigma,(0,T]} + 0\\
&\geq 0.
\end{align*}
Equivalently,
$$\langle\mathcal{S}(\mathbf{u}|_\Gamma),\mathbf{v}-\mathbf{u}\rangle_{0,\Gamma_\Sigma,(0,T]} \geq \langle\mathbf{f},\mathbf{v}-\mathbf{u}\rangle_{0,\Gamma_\Sigma,(0,T]},$$
and \eqref{eq:VarIneq} follows.\\

To show that, conversely, a sufficiently smooth solution ${\bf u}_{|_\Gamma}$ of \eqref{eq:VarIneq} leads to a solution ${\bf u}$ of the contact problem \eqref{contactbc}-\eqref{prob:strong_formulation},
we first recall that the representation formula \eqref{representation formula} recovers the solution $\mathbf{u}$ to \eqref{strongformPDE} in the domain $[0,T]\times \Omega$ from its restriction to $[0,T]\times \Gamma$. It remains to show that $\mathbf{u}$ satisfies the boundary conditions on $\Gamma$. Because $\textbf{u} \in \mathcal{C}\subset H^{1/2}([0,T],\tilde{H}^{1/2}(\Gamma_\Sigma))^d$, the homogeneous Dirichlet condition is satisfied on $\Gamma_D$. 

We now choose $\mathbf{v} = \mathbf{u}|_\Gamma + \mathbf{w} \in \mathcal{C}$, for $\mathbf{w}\in H^{1/2}([0,T],\tilde{H}^{1/2}({\Gamma_\Sigma}))^d$ with $w_\perp =0$ a.e.~on $[0,T]\times \Gamma_C$. Then \eqref{eq:VarIneq} implies
$$\langle\mathcal{S}(\mathbf{u}|_\Gamma),\mathbf{w}\rangle_{0,\Gamma_\Sigma,(0,T]} \geq \langle\mathbf{f},\mathbf{w}\rangle_{0,\Gamma_\Sigma,(0,T]}.$$
Similarly, $\mathbf{v} = \mathbf{u}|_\Gamma - \mathbf{w} \in \mathcal{C}$ and therefore also
$$\langle\mathcal{S}(\mathbf{u}|_\Gamma),\mathbf{w}\rangle_{0,\Gamma_\Sigma,(0,T]} \leq \langle\mathbf{f},\mathbf{w}\rangle_{0,\Gamma_\Sigma,(0,T]}.$$
We conclude
$$\langle\mathcal{S}(\mathbf{u}|_\Gamma),\mathbf{w}\rangle_{0,\Gamma_\Sigma,(0,T]} = \langle\mathbf{f},\mathbf{w}\rangle_{0,\Gamma_\Sigma,(0,T]}. $$
As this holds for all $\mathbf{w}\in H^{1/2}([0,T],\tilde{H}^{1/2}({\Gamma_\Sigma}))^d$ with $w_\perp =0$ a.e.~on $[0,T]\times \Gamma_C$, we conclude $\mathbf{p}= \mathcal{S}(\mathbf{u}|_\Gamma)=\mathbf{f}$ on $\Gamma_N$ and $p_\parallel= \mathcal{S}(\mathbf{u}|_\Gamma)_\parallel=f_\parallel$ on $\Gamma_C$.

It remains to show the Signorini conditions on $\Gamma_C$. For this we note that $u_\perp\geq g$ on $\Gamma_C$ because $\mathbf{u}|_\Gamma \in \mathcal{C}$. We now choose $\mathbf{v} = \mathbf{u}|_\Gamma + \mathbf{w} \in \mathcal{C}$ with $w_\perp \geq 0$ on $\Gamma_C$. As we have already concluded $\mathcal{S}(\mathbf{u}|_\Gamma)_\parallel=f_\parallel$, we  obtain 
$$0\geq \langle-\mathcal{S}(\mathbf{u}|_\Gamma)+\mathbf{f},\mathbf{w}\rangle_{0,\Gamma_\Sigma,(0,T]} = \langle-\mathcal{S}(\mathbf{u}|_\Gamma)_\perp+f_\perp,w_\perp\rangle_{0,\Gamma_\Sigma,(0,T]}.$$
Therefore the boundary condition $\mathcal{S}(\mathbf{u}|_\Gamma)_\perp\geq f_\perp$ is satisfied almost everywhere on $\Gamma_C$. 

Finally, for every $\mathbf{v}\in \mathcal{C}$ with $\mathbf{v}|_{\Gamma_C} = g\mathbf{n}$,  we similarly note  
\begin{equation}\label{eq:aux}
0\geq \langle-\mathcal{S}(\mathbf{u}|_\Gamma)_\perp+f_\perp,g-u_\perp\rangle_{0,\Gamma_\Sigma,(0,T]}.  \end{equation}
In the previous step we saw that $0\geq -\mathcal{S}(\mathbf{u}|_\Gamma)_\perp+f_\perp$ almost everywhere, while $0\geq g-u_\perp$ for $\mathbf{u}|_\Gamma  \in \mathcal{C}$. Hence also $0\leq (-\mathcal{S}(\mathbf{u}|_\Gamma)_\perp+f_\perp)(g-u_\perp)$ almost everywhere. Together with \eqref{eq:aux} we conclude that $$0= (-\mathcal{S}(\mathbf{u}|_\Gamma)_\perp+f_\perp)(g-u_\perp)$$
almost everywhere on $\Gamma_C$, i.e. the equality in \eqref{scontactbc2}.

Therefore, $\mathbf{u}$ satisfies the boundary conditions of the contact problem \eqref{contactbc}-\eqref{prob:strong_formulation}.
\end{proof}

In spite of the interest in the dynamic contact problem \eqref{contactbc}-\eqref{prob:strong_formulation}, from a rigorous mathematical perspective the existence of solutions is only known for certain materials with dissipation or for not perfectly rigid, but dissipative obstacles, see e.g.~\cite{cocou}. Without dissipation the existence of a solution has been proven for simplified problems involving the scalar wave equation in special geometries \cite{cooper, lebeau}. In these cases refined information about the coercivity of the Poincar\'{e}-Steklov operator is available, as discussed in \cite{contact}.\\

For the numerical approximation we will consider a mixed formulation. Let
\begin{align}
M^+:=\left\lbrace {\mu} \in H^{1/2}([0,T],\tilde{H}^{-1/2}(\Gamma_C)): \left\langle {\mu},{w}\right\rangle_{0,\Gamma_C,(0,T]} \geq 0, \ \forall \ 0\leq { w} \in H^{-\frac{1}{2}}([0,T],\tilde{H}^{1/2}(\Gamma_\Sigma)) \right\rbrace
\end{align}
be the set of admissible Lagrange multipliers, in which the representative 
\begin{equation}
{\lambda}= (\mathcal{S}{\bf u}- {\bf f})_{\perp} \in M^+
\end{equation}
is sought. Note that by definition ${ \lambda} \geq 0$ and $\lambda = 0$ outside $\Gamma_C$. Then, the mixed formulation reads:\\

\textit{find} $(\textbf{u},{\lambda}) \in H^{1/2}([0,T],\tilde{H}^{1/2}(\Gamma_\Sigma))^d \times M^+$ \textit{such that} 
\begin{subequations} \label{eq:MixedProblem}
 \begin{alignat}{2}
\left\langle \mathcal{S}\textbf{u},\textbf{v} \right\rangle_{0,\Gamma_\Sigma,(0,T]} - \left\langle \lambda,v_\perp \right\rangle_{0,\Gamma_C,(0,T]} &= \left\langle \textbf{f},\textbf{v} \right\rangle_{0,\Gamma_\Sigma,(0,T]} &\quad &\forall \textbf{v}\in H^{1/2}([0,T],\tilde{H}^{1/2}(\Gamma_\Sigma))^d \label{eq:WeakMixedVarEq}\\
\left\langle {\mu} -{\lambda},u_\perp \right\rangle_{0,\Gamma_C,(0,T]} & \geq \left\langle g,\mu-\lambda \right\rangle_{0,\Gamma_C,(0,T]} &\quad &\forall {\mu} \in M^+ .\label{eq:ContContactConstraints}
\end{alignat}
\end{subequations}
\begin{theorem}
The mixed formulation \eqref{eq:MixedProblem}  is equivalent to the variational inequality \eqref{eq:VarIneq}.
\end{theorem}
\begin{proof}The proof is identical to the proof of Theorem 14 in \cite{contact} for the wave equation.
\end{proof}

\section{Discretization}\label{sec 3}

To solve 
the mixed formulation \eqref{eq:MixedProblem}, in a discretized form,  we  consider a uniform decomposition of the time interval $[0,T]$ with time step $\Delta t=\frac{T}{N_{\Delta t}}$, $N_{\Delta t}\in\mathbb{N}^{+}$, generated by the time instants $t_{\ell}=\ell\Delta t$, $\ell=0,\ldots,N_{\Delta t}$. We define the corresponding spaces
\begin{equation}\label{time_set}
\begin{array}{ll}
 V^{-1}_{\Delta t}=&\left\lbrace v_{\Delta t}\in L^2([0,T])
\: : \: {v_{\Delta t}}_{\vert_{[t_\ell,t_{\ell+1}]}}\in \mathcal{P}_0, \:\forall \ell=0,..., N_{\Delta t}-1 \right\rbrace ,\\[4pt]
V^0_{\Delta t}=&\left\lbrace r_{\Delta t}\in C^0([0,T])
\: : \: {r_{\Delta t}}_{\vert_{[t_\ell,t_{\ell+1}]}}\in \mathcal{P}_1, \:\forall \ell=0,..., N_{\Delta t}-1,\: v(0)=0  \right\rbrace,
\end{array}
\end{equation}
where $\mathcal{P}_s$, $s\geq 0$, is the space of the algebraic polynomials of degree s. For the space discretization with $d=2$, we introduce a boundary mesh constituted by a set of straight line segments $\mathcal{T}=\left\lbrace e_1,... ,e_M \right\rbrace$ such that $h_i:=length(e_i)\leqslant h$, $e_i\cap e_j=\emptyset$ if $i\neq j$ and $\cup_{i=1}^M \overline{e}_i=\overline{\Gamma}$ if $\Gamma$ is polygonal, or a suitably fine approximation of $\Gamma$ otherwise. For $d=3$, we assume that $\Gamma$ is triangulated by  $\mathcal{T}=\{e_1,\cdots,e_{M}\}$, with $h_i:=diam(e_i)\leqslant h$, $e_i\cap e_j=\emptyset$ if $i\neq j$ and, if $\overline{e_i}\cap \overline{e_j} \neq \emptyset$, the intersection is either an edge or a vertex of both triangles. We denote by $h$ the maximum of the $h_i$.
On $\mathcal{T}$  we consider the spaces of piecewise polynomial functions 
\begin{equation}\label{pol_space_4_discontinuous_f}
X^{-1}_{h,\Gamma}=\left\lbrace w_h\in L^2(\Gamma)\: : \: w_h\vert_{e_i}\in \mathcal{P}_s,\: e_i\in \mathcal{T}  \right\rbrace\subset H^{-1/2}(\Gamma), \quad X^{-1}_{h,\Gamma'} = X^{-1}_{h,\Gamma} \cap \widetilde{H}^{-1/2}(\Gamma'), 
\end{equation}
\begin{equation}\label{pol_space_4_continuous_f}
X^{0}_{h,\Gamma}=\left\lbrace w_h\in C^0(\Gamma)\: : \: w_h\vert_{e_i}\in \mathcal{P}_s,\: e_i\in \mathcal{T}  \right\rbrace \subset H^{1/2}(\Gamma), \quad X^{0}_{h,\Gamma'} = X^{0}_{h,\Gamma} \cap \widetilde{H}^{1/2}(\Gamma'),
\end{equation}
where $\Gamma' \subset \Gamma$. \\
The full discretization of \eqref{eq:MixedProblem} involves the subspace
\begin{align}
M^+_{H,\Delta T}:=\left\lbrace {\mu}_{H,\Delta T} \in X^{-1}_{H,\Gamma_C}\otimes V^{-1}_{\Delta T}: \mu_{H,\Delta T} \geq 0 \text{ on } \Gamma_C \times [0,T] \right\rbrace \subset M^+
\end{align}
and {a discretized version ${\cal S}_{h,\Delta t}$ of the Poincar\'e-Steklov operator. Two specific discretizations ${\cal S}_{h,\Delta t}$ will be discussed in detail in Section \ref{subPS}, corresponding to  
 the symmetric and the non-symmetric formulations \eqref{definition_S}, \eqref{definition_S_symm} of the operator
${\cal S}$. The full discretization of \eqref{eq:MixedProblem} then reads:}\\

\textit{find} $(\textbf{u}_{h,\Delta t},{\lambda}_{H,\Delta T}) \in (X^0_{h,\Gamma_\Sigma}\otimes V^0_{\Delta t})^d \times M^+_{H,\Delta T}$ \textit{such that}
\begin{subequations} \label{eq:MixedProblemh}
\begin{alignat}{2}
 &\left\langle \mathcal{S}_{h,\Delta t}\textbf{u}_{h,\Delta t},\textbf{v}_{h,\Delta t} \right\rangle_{0,\Gamma_\Sigma,(0,T]} -\left\langle {\lambda}_{H,\Delta T},{v}_{\perp,h,\Delta t} \right\rangle_{0,\Gamma_C,(0,T]} = \left\langle \textbf{f},\textbf{v}_{h,\Delta t} \right\rangle_{0,\Gamma_\Sigma,(0,T]} \quad\forall \textbf{v}_{h,\Delta t} \in (X^0_{h,\Gamma_\Sigma}\otimes V^0_{\Delta t})^d \label{eq:WeakMixedVarEqh} \\
&\left\langle {\mu}_{H,\Delta T} -{\lambda}_{H,\Delta T},{u}_{\perp,h,\Delta t} \right\rangle_{0,\Gamma_C,(0,T]}  \geq \left\langle g,\mu_{H,\Delta T}-\lambda_{H,\Delta T} \right\rangle_{0,\Gamma_C,(0,T]}  \quad\quad\quad\quad\quad\quad\:\forall {\mu}_{H,\Delta T} \in M^+_{H,\Delta T}. \label{eq:ContContactConstraintsh}
\end{alignat}
\end{subequations}

\begin{theorem}
The discretized mixed formulation \eqref{eq:MixedProblemh} admits a unique solution provided that ${\cal S}_{h,\Delta t}$ is positive definite. 
\end{theorem}
\begin{proof}
As ${\cal S}_{h,\Delta t}$ is positive definite, the existence of a unique solution follows from standard results for saddle point problems.
\end{proof}
Theoretically, it is not generally known that ${\cal S}_{h,\Delta t}$ is positive definite, a main difficulty in the analysis of contact problems as discussed in \cite{contact}. Empirically, we obtain a well-conditioned, discretized mixed formulation with, in particular, a unique solution.\\

A standard solver for the discrete formulation \eqref{eq:MixedProblemh} is given by the Uzawa algorithm. This is described in the following and it involves the $L^2$-projection $\text{Pr}_C : L^2((0,T]; L^2(\Gamma ) ) \to M^+_{H,\Delta T}$: 
\begin{algorithm}[H]
\caption{(Uzawa algorithm)}
\label{alg1}
\begin{algorithmic}
\STATE Fix $\rho>0$.
\STATE $k=0$,  $\lambda_{H,\Delta T}^{(0)}= 0$
\WHILE{stopping criterion not satisfied}
\STATE \textbf{solve}$\quad$ equation \eqref{eq:WeakMixedVarEqh} for ${\bf u}_{h,\Delta t}^{(k)}$
\STATE \textbf{compute}$\quad$ ${\lambda}^{(k+1)}_{H,\Delta T}= \text{Pr}_C ({\lambda}_{H,\Delta T}^{(k)} -\rho ({ u}_{\perp,h,\Delta t}^{(k)}-g)) $ 
\STATE $k \leftarrow k+1$
\ENDWHILE
\end{algorithmic}
\end{algorithm}

\begin{lemma}
Assume that ${\cal S}_{h,\Delta t}$ is positive definite. The space-time Uzawa algorithm converges, provided that $0<\rho< 2C$, with $C$ the smallest eigenvalue of ${\cal S}_{h,\Delta t}$. 
\end{lemma}
\begin{proof}
Note that $\lambda_{H,\Delta T} = \text{Pr}_C\ {\lambda}_{H,\Delta T} =\text{Pr}_C\ ({\lambda}_{H,\Delta T} -\rho ({ u}_{\perp,h,\Delta t}-g))$, because $\lambda_{H,\Delta T} \in M^+_{H,\Delta T}$ and $\rho ({ u}_{\perp,h,\Delta t}-g) \geq 0$. From the algorithm and the contraction property of the projection $\textrm{Pr}_C$, one has
\begin{align*}
\|{\lambda}_{H,\Delta T}^{(k+1)} - {\lambda}_{H,\Delta T}\|_{L^2([0,T],L^2(\Gamma_C))}^2 &= \|\text{Pr}_C\ ({\lambda}_{H,\Delta T}^{(k)} -\rho ({u}_{\perp,h,\Delta t}^{(k)}-g)) - \text{Pr}_C\ ({\lambda}_{H,\Delta T} -\rho ({ u}_{\perp,h,\Delta t}-g))\|_{L^2([0,T],L^2(\Gamma_C))}^2\\
& \leq \|{\lambda}_{H,\Delta T}^{(k)} -{\lambda}_{H,\Delta T} -\rho ( {u}_{\perp,h,\Delta t}^{(k)} - {u}_{\perp,h,\Delta t})\|_{L^2([0,T],L^2(\Gamma_\Sigma))}^2\\
& = \|{\lambda}_{H,\Delta T}^{(k)} -{\lambda}_{H,\Delta T}\|_{L^2([0,T],L^2(\Gamma_C))}^2  + \rho^2 \|{u}_{\perp,h,\Delta t}^{(k)} - {u}_{\perp,h,\Delta t}\|^2_{L^2([0,T],L^2(\Gamma_\Sigma))} \\ & \qquad
- 2 \rho \langle {\lambda}_{H,\Delta T}^{(k)}-{\lambda}_{H,\Delta T} , {u}_{\perp,h,\Delta t}^{(k)} - {u}_{\perp,h,\Delta t}\rangle_{0,\Gamma_C,(0,T]}
\ .
\end{align*}
Using $\|{u}_{\perp,h,\Delta t}^{(k)} - {u}_{\perp,h,\Delta t}\|_{L^2([0,T],L^2(\Gamma_\Sigma))} \leq \|{\bf{u}}_{h,\Delta t}^{(k)} - {\bf{u}}_{h,\Delta t}\|_{L^2([0,T],L^2(\Gamma_\Sigma))^d}$, we find that 
\begin{align*}
&\|{\lambda}_{H,\Delta T}^{(k)} - {\lambda}_{H,\Delta T}\|_{L^2([0,T],L^2(\Gamma_C))}^2 - \|{\lambda}_{H,\Delta T}^{(k+1)} - {\lambda}_{H,\Delta T}\|_{L^2([0,T],L^2(\Gamma_C))}^2 \\& \geq 2 \rho \langle  {\lambda}_{H,\Delta T}^{(k)} -{\lambda}_{H,\Delta T}, {u}_{\perp,h,\Delta t}^{(k)} - {u}_{\perp,h,\Delta t}\rangle_{0,\Gamma_C,(0,T]} - \rho^2 \|{\bf{u}}_{h,\Delta t}^{(k)} - {\bf{u}}_{h,\Delta t}\|_{L^2([0,T],L^2(\Gamma_\Sigma))^d}^2\ .
\end{align*}
Further note 
\begin{align*}
\langle {\lambda}_{H,\Delta T}^{(k)} -{\lambda}_{H,\Delta T} , {u}_{\perp,h,\Delta t}^{(k)} - {u}_{\perp,h,\Delta t}\rangle_{0,\Gamma_C,(0,T]}  &= \langle \mathcal{S}_{h,\Delta t}({\bf{u}}_{h,\Delta t}^{(k)} - {\bf{u}}_{h,\Delta t}) , {\bf{u}}_{h,\Delta t}^{(k)} - {\bf{u}}_{h,\Delta t}\rangle_{0,\Gamma_\Sigma,(0,T]}\\
&\geq C\|{\bf{u}}_{h,\Delta t}^{(k)} - {\bf{u}}_{h,\Delta t}\|_{L^2([0,T],L^2(\Gamma_\Sigma))^d}^2\ ,
\end{align*}
so that  $$\|{\lambda}_{H,\Delta T}^{(k)} - {\lambda}_{H,\Delta T}\|_{L^2([0,T],L^2(\Gamma_C))}^2 - \|{\lambda}_{H,\Delta T}^{(k+1)} - {\lambda}_{H,\Delta T}\|_{L^2([0,T],L^2(\Gamma_C))}^2 \geq (2 \rho\, C - \rho^2) \|{\bf{u}}_{h,\Delta t}^{(k)} - {\bf{u}}_{h,\Delta t}\|^2_{L^2([0,T],L^2(\Gamma_\Sigma))^d}\ .$$ If $0<\rho<2C$ If $0<\rho<2C$, the right hand side is non-negative, and $\|{\lambda}_{H,\Delta T}^{(k)} - {\lambda}_{H,\Delta T}\|_{L^2([0,T],L^2(\Gamma_C))}$ is a decreasing sequence. As $\|{\lambda}_{H,\Delta T}^{(k)} - {\lambda}_{H,\Delta T}\|_{L^2([0,T],L^2(\Gamma_C))}\geq 0$, it converges, and we conclude  $ \|{\bf{u}}_{h,\Delta t}^{(k)} - {\bf{u}}_{h,\Delta t}\|_{L^2([0,T],L^2(\Gamma_\Sigma))^d} \to 0$.  
\end{proof}
In addition to the Uzawa algorithm for the space-time problem, as above, also Uzawa algorithms in each time-step were studied in \cite{contact}. They provide an alternative solver for problem \eqref{eq:MixedProblemh}.

\section{Error estimates}\label{sec4}

In this section we show an a priori estimate for the mixed formulation of the contact problem \eqref{eq:MixedProblem} and its discretization \eqref{eq:MixedProblemh}. It builds on the corresponding analysis for the scalar wave equation in \cite{contact} where, however, refined information about the Poincar\'{e}-Steklov operator is available. Building on the classical theory for time-domain boundary integral equations \cite{hd}, the analysis requires the  formulation of \eqref{eq:MixedProblem} presented for times $t \in \mathbb{R}^+$, using the generalized inner product \eqref{sigma_product} for $\sigma>0$ and considering
\begin{align}
{\widetilde{M}}^+:=\left\lbrace {\mu} \in H^{1/2}_\sigma(\mathbb{R}^+,\tilde{H}^{-1/2}(\Gamma_C)): \left\langle {\mu},{w}\right\rangle_{\sigma,\Gamma_C,\mathbb{R}^+} \geq 0, \ \forall \ 0\leq { w} \in H^{-1/2}_\sigma(\mathbb{R}^+,\tilde{H}^{1/2}(\Gamma_\Sigma)) \right\rbrace\,.
\end{align}
For $\textbf{f} \in H^{1/2}_\sigma\left(\mathbb{R}^+,H^{-1/2}(\Gamma_\Sigma)\right)^d$ and $g \in H^{1/2}_\sigma(\mathbb{R}^+,H^{1/2}(\Gamma_C))$ the formulation is given by:\\

\noindent \textit{Find} $(\textbf{u},{\lambda}) \in H^{1/2}_\sigma(\mathbb{R}^+,\tilde{H}^{1/2}(\Gamma_\Sigma))^d \times {\widetilde{M}}^+$ \textit{such that} 
\begin{subequations} \label{eq:MixedProblemsigma}
 \begin{alignat}{2}
\left\langle \mathcal{S}\textbf{u},\textbf{v} \right\rangle_{\sigma,\Gamma_\Sigma,\mathbb{R}^+} - \left\langle \lambda,v_\perp \right\rangle_{\sigma,\Gamma_C,\mathbb{R}^+} &= \left\langle \textbf{f},\textbf{v} \right\rangle_{\sigma,\Gamma_\Sigma,\mathbb{R}^+} &\quad &\forall \textbf{v}\in H^{1/2}_\sigma(\mathbb{R}^+,\tilde{H}^{1/2}(\Gamma_\Sigma))^d \label{eq:WeakMixedVarEqsigma}\\
\left\langle {\mu} -{\lambda},u_\perp \right\rangle_{\sigma,\Gamma_C,\mathbb{R}^+} & \geq \left\langle g,\mu-\lambda \right\rangle_{\sigma,\Gamma_C,\mathbb{R}^+} &\quad &\forall {\mu} \in {\widetilde{M}}^+ .\label{eq:ContContactConstraintssigma}
\end{alignat}
\end{subequations}

The corresponding discretization, generalizing \eqref{eq:MixedProblemh}, is given by:\\

\noindent \textit{Find} $(\textbf{u}_{h,\Delta t},{\lambda}_{H,\Delta T}) \in (X^0_{h,\Gamma_\Sigma}\otimes V^0_{\Delta t})^d \times M^+_{H,\Delta T}$ \textit{such that}
\begin{subequations} \label{eq:MixedProblemhsigma}
\begin{alignat}{2}
 &\left\langle \mathcal{S}_{h,\Delta t}\textbf{u}_{h,\Delta t},\textbf{v}_{h,\Delta t} \right\rangle_{\sigma,\Gamma_\Sigma,\mathbb{R}^+} - \left\langle {\lambda}_{H,\Delta T},{v}_{\perp,h,\Delta t} \right\rangle_{\sigma,\Gamma_C,\mathbb{R}^+} = \left\langle \textbf{f},\textbf{v}_{h,\Delta t} \right\rangle_{\sigma,\Gamma_\Sigma,\mathbb{R}^+} \quad\forall \textbf{v}_{h,\Delta t} \in (X^0_{h,\Gamma_\Sigma}\otimes V^0_{\Delta t})^d \label{eq:WeakMixedVarEqhsigma} \\
&\left\langle {\mu}_{H,\Delta T} -{\lambda}_{H,\Delta T},{u}_{\perp,h,\Delta t} \right\rangle_{\sigma,\Gamma_C,\mathbb{R}^+}  \geq \left\langle g,\mu_{H,\Delta T}-\lambda_{H,\Delta T} \right\rangle_{\sigma,\Gamma_C,\mathbb{R}^+}  \quad\quad\quad\quad\,\,\:\forall {\mu}_{H,\Delta T} \in M^+_{H,\Delta T}\ , \label{eq:ContContactConstraintshsigma}
\end{alignat}
\end{subequations}
where the involved discrete functional spaces contain functions supported in a finite number of time steps.\\
For the proof, we recall basic mapping properties of the Poincar\'{e}-Steklov operator $\mathcal{S}$. They are summarized in the following Theorem \ref{mappingProperties}. See \cite{cooper,contact} for corresponding results for the wave equation, with analogous proofs. 
\begin{theorem}
\label{mappingProperties} For $r \in \mathbb{R}$ and $\sigma>0$, $\mathcal{S}: H^{r}_\sigma(\mathbb{R}^+, {H}^{\frac{1}{2}}(\Gamma))^d \to H^{r}_\sigma(\mathbb{R}^+, {H}^{-\frac{1}{2}}(\Gamma))^d$ continuously.
\end{theorem}
Here and below we write $A \lesssim B$ provided there exists a constant $C$ such that $A \leq CB$. If the constant $C$ is allowed to depend on a parameter $\sigma$, we write $A \lesssim_\sigma B$.\\

Like in \cite{contact} an inf-sup estimate is a crucial ingredient to estimate the error: 
\begin{theorem}\label{discInfSup}
Let $C>0$ sufficiently small and $\frac{\max\{h, \Delta t\}}{\min\{H, \Delta T\}}<C$. Then there exists $\alpha>0$ such that 
$\forall\,\lambda_{H, \Delta T} \in  X^{-1}_{H,\Gamma_C}\otimes V^{-1}_{\Delta T}$: 
$$\sup_{\mathbf{v}_{h,\Delta t} \in (X^0_{h,\Gamma_\Sigma}\otimes V^0_{\Delta t})^d} \frac{\langle v_{\perp,h,\Delta t}, \lambda_{ H, \Delta T}\rangle_{\sigma, \Gamma_C, \mathbb{R}^+}}{\|\mathbf{v}_{h,\Delta t}\|_{ 0,\frac{1}{2}, \sigma, \ast}} \geq \alpha \|\lambda_{H,\Delta T}\|_{0, -\frac{1}{2}, \sigma }\ .$$
\end{theorem}
The proof is found in \cite[Theorem 15]{contact}. Now, we can prove the following result:
\begin{theorem}\label{apriori}
Let $(\textbf{u},{\lambda}) \in H^{1/2}_\sigma(\mathbb{R}^+,\tilde{H}^{1/2}(\Gamma_\Sigma))^d \times {\widetilde{M}}^+$ be a solution to the mixed problem \eqref{eq:MixedProblemsigma} and $(\textbf{u}_{h,\Delta t},{\lambda}_{H,\Delta T}) \in (X^0_{h,\Gamma_\Sigma}\otimes V^0_{\Delta t})^d \times M^+_{H,\Delta T}$ a solution of the discretized mixed problem \eqref{eq:MixedProblemhsigma}. Assume that $\mathcal{S}$ is coercive. Then for a sufficiently small constant $C>0$ and $\frac{\max\{h, \Delta t\}}{\min\{H, \Delta T\}}<C$, the following a priori estimates hold: \\
\begin{align}
\| \lambda -\lambda_{H,\Delta T} \|_{0,-\frac{1}{2}, \sigma}  &\lesssim_\sigma \inf \limits_{\tilde{\lambda}_{H,\Delta T} \in M^+_{H,\Delta T}} \| \lambda - \tilde{\lambda}_{H,\Delta T} \|_{0,-\frac{1}{2},\sigma} +(\Delta t)^{-\frac{1}{2}}  \|{\bf u}- {\bf u}_{h,\Delta t}\|_{-\frac{1}{2},\frac{1}{2},\sigma, \ast} \ ,\label{est_1}\\
\|{\bf u}-{\bf u}_{h,\Delta t}\|_{-\frac{1}{2},\frac{1}{2},\sigma, \ast} &\lesssim_\sigma \inf \limits_{{\bf v}_{h,\Delta t} \in (X^0_{h,\Gamma_\Sigma}\otimes V^0_{\Delta t})^d}
\|{\bf u}-{\bf v}_{h,\Delta t}\|_{\frac{1}{2},\frac{1}{2},\sigma, \ast}\nonumber \\& \qquad
+ \inf \limits_{\tilde{\lambda}_{H,\Delta T} \in M^+_{H,\Delta T}}\left\{\|\tilde{\lambda}_{H,\Delta T} - \lambda\|_{\frac{1}{2},-\frac{1}{2},\sigma} +\|\tilde{\lambda}_{H,\Delta T} - {\lambda}_{H,\Delta T}\|_{\frac{1}{2},-\frac{1}{2},\sigma}\right\} \ \label{est_2}. 
\end{align}
\end{theorem}
\begin{proof}
Using equations \eqref{eq:MixedProblemsigma} and \eqref{eq:MixedProblemhsigma}, we note that for arbitrary $\tilde{\lambda}_{H,\Delta T} \in M^+_{H,\Delta T}$ the following identity holds:
\begin{align} \label{lambdaEq}
\langle \lambda_{H,\Delta T}-\tilde{\lambda}_{H,\Delta T}, v_{\perp,h,\Delta t} \rangle_{\sigma,\Gamma_C,\mathbb{R}^+} & = \langle \mathcal{S} \mathbf{u}_{h,\Delta t}, \mathbf{v}_{h,\Delta t} \rangle_{\sigma,\Gamma_\Sigma,\mathbb{R}^+} - \langle \mathbf{f}, \mathbf{v}_{h,\Delta t}\rangle_{\sigma,\Gamma_\Sigma,\mathbb{R}^+} - \langle \tilde{\lambda}_{H,\Delta T}, v_{\perp,h,\Delta t} \rangle_{\sigma,\Gamma_C,\mathbb{R}^+} \nonumber \\  &= \langle \mathcal{S} \mathbf{u}_{h,\Delta t}, \mathbf{v}_{h,\Delta t} \rangle_{\sigma,\Gamma_\Sigma,\mathbb{R}^+} - \langle \mathcal{S} \mathbf{u}, \mathbf{v}_{h,\Delta t} \rangle_{\sigma,\Gamma_\Sigma,\mathbb{R}^+} +\langle  \lambda , v_{\perp,h,\Delta t} \rangle_{\sigma,\Gamma_C,\mathbb{R}^+}\nonumber \\ &\qquad -\langle \tilde{\lambda}_{H,\Delta T}, v_{\perp,h,\Delta t} \rangle_{\sigma,\Gamma_C,\mathbb{R}^+} \nonumber 
\\ &= \langle \mathcal{S} (\mathbf{u}_{h,\Delta t} -\mathbf{u}),\mathbf{v}_{h,\Delta t} \rangle_{\sigma,\Gamma_\Sigma,\mathbb{R}^+} + \langle \lambda - \tilde{\lambda}_{H,\Delta T}, v_{\perp,h,\Delta t} \rangle_{\sigma,\Gamma_C,\mathbb{R}^+}.
\end{align} 
From \eqref{lambdaEq} and the inf-sup condition in Theorem \ref{discInfSup}, we obtain:
\begin{align*}
\alpha \| \lambda_{H,\Delta T}-\tilde{\lambda}_{H,\Delta T}\|_{0,-\frac{1}{2},\sigma} &\leq \sup_{ \mathbf{v}_{h,\Delta t}\in (X^0_{h,\Gamma_\Sigma}\otimes V^0_{\Delta t})^d} \frac{ \langle  \lambda_{H,\Delta T}-\tilde{\lambda}_{H,\Delta T}, v_{\perp,h,\Delta t} \rangle_{\sigma,\Gamma_C,\mathbb{R}^+}}{\|\mathbf{v}_{h,\Delta t}\|_{0,\frac{1}{2}, \sigma, \ast} }
\\ & =
\sup_{ \mathbf{v}_{h,\Delta t}\in (X^0_{h,\Gamma_\Sigma}\otimes V^0_{\Delta t})^d} \frac{\langle \mathcal{S} (\mathbf{u}_{h,\Delta t} -\mathbf{u}),\mathbf{v}_{h,\Delta t} \rangle_{\sigma,\Gamma_\Sigma,\mathbb{R}^+} + \langle \lambda-\tilde{\lambda}_{H,\Delta T}, v_{\perp,h,\Delta t} \rangle_{\sigma,\Gamma_C,\mathbb{R}^+}}{\|\mathbf{v}_{h,\Delta t}\|_{0,\frac{1}{2}, \sigma, \ast}}\ .
\end{align*}
We separately estimate the two terms of the numerator. Using the continuity of the  duality pairing and an inverse inequality in time \cite{setup}, the first term is estimated by\\
\begin{align*}
|\langle \mathcal{S} (\mathbf{u}_{h,\Delta t} -u),\mathbf{v}_{h,\Delta t} \rangle_{\sigma,\Gamma_\Sigma,\mathbb{R}^+}| & \leq \| \mathcal{S}(\mathbf{u}_{h,\Delta t}-\mathbf{u})\|_{-\frac{1}{2},-\frac{1}{2},\sigma} \|{\mathbf{v}_{h,\Delta t}}\|_{\frac{1}{2},\frac{1}{2},\sigma, \ast} 
\\ &\lesssim\,
\| \mathbf{u}_{h,\Delta t} -\mathbf{u}\|_{-\frac{1}{2},\frac{1}{2},\sigma,\ast}(\Delta t)^{-\frac{1}{2}} \|\mathbf{v}_{h,\Delta t}\|_{0,\frac{1}{2},\sigma, \ast}\ .\end{align*}
For the second term, a similar argument yields:
\begin{align*}
|\langle \lambda-\tilde{\lambda}_{H,\Delta T}, v_{\perp,h,\Delta t} \rangle_{\sigma,\Gamma_C,\mathbb{R}^+}|& \leq \| \lambda-\tilde{\lambda}_{H,\Delta T} \|_{0,-\frac{1}{2},\sigma} \|\mathbf{v}_{h,\Delta t}\|_{0,\frac{1}{2},\sigma, \ast}
\ .
\end{align*} 
We conclude the a priori estimate as in \eqref{est_1}, i.e.
\begin{align}\label{lambdaapriori}
 \|\lambda -\lambda_{H,\Delta T}\|_{0,-\frac{1}{2}, \sigma} &\leq \inf \limits_{\tilde{\lambda}_{H,\Delta T}}  \left( \|\lambda -\tilde{\lambda}_{H,\Delta T}\|_{0,-\frac{1}{2}, \sigma} + \|{\lambda}_{H,\Delta T} -\tilde{\lambda}_{H,\Delta T}\|_{0,-\frac{1}{2}, \sigma} \right) \nonumber \\  &\lesssim_\sigma  \inf \limits_{\tilde{\lambda}_{H,\Delta T}} \| \lambda - \tilde{\lambda}_{H,\Delta T} \|_{0,-\frac{1}{2},\sigma} +(\Delta t)^{-\frac{1}{2}}  \| \mathbf{u}_{h,\Delta t}-\mathbf{u}\|_{-\frac{1}{2},\frac{1}{2},\sigma, \ast} \ .\nonumber
\end{align}
Next we combine the Galerkin orthogonality
\begin{equation*}
\langle\mathcal{S}(\mathbf{u}-\mathbf{u}_{h,\Delta t}),\mathbf{v}_{h,\Delta t}\rangle_{\sigma,\Gamma_\Sigma,\mathbb{R}^+} = \langle \lambda -{\lambda}_{H,\Delta T}, v_{\perp,h,\Delta t}\rangle_{\sigma,\Gamma_C,\mathbb{R}^+}
\end{equation*}
with the coercivity of the Poincar\'{e}-Steklov operator, leading to
\begin{align*}
\|\mathbf{u}_{h,\Delta t} - \mathbf{v}_{h,\Delta t}\|_{-\frac{1}{2},\frac{1}{2},\sigma, \ast}^2 &\lesssim_\sigma\langle\mathcal{S}(\mathbf{u}_{h,\Delta t}-\mathbf{v}_{h,\Delta t}), \mathbf{u}_{h,\Delta t}- \mathbf{v}_{h,\Delta t}\rangle_{\sigma,\Gamma_\Sigma,\mathbb{R}^+} \\
& = \langle\mathcal{S}(\mathbf{u}-\mathbf{v}_{h,\Delta t}),\mathbf{u}_{h,\Delta t}- \mathbf{v}_{h,\Delta t}\rangle_{\sigma,\Gamma_\Sigma,\mathbb{R}^+} + \langle\mathcal{S}(\mathbf{u}_{h,\Delta t}-\mathbf{u}), \mathbf{u}_{h,\Delta t}- \mathbf{v}_{h,\Delta t}\rangle_{\sigma,\Gamma_\Sigma,\mathbb{R}^+}\\ 
& = \langle \mathcal{S}(\mathbf{u}-\mathbf{v}_{h,\Delta t}),\mathbf{u}_{h,\Delta t}- \mathbf{v}_{h,\Delta t}\rangle_{\sigma,\Gamma_\Sigma,\mathbb{R}^+}\\& \qquad+\langle \tilde{\lambda}_{H,\Delta T} - \lambda+{\lambda}_{H,\Delta T}-\tilde{\lambda}_{H,\Delta T}, u_{\perp,h,\Delta t}-  v_{\perp,h,\Delta t}\rangle_{\sigma,\Gamma_C,\mathbb{R}^+} 
\end{align*}
for all $\mathbf{v}_{h,\Delta t}$ and $\tilde{\lambda}_{H,\Delta T}$. The mapping properties of $\mathcal{S}$ and the continuity of the duality pairing then show
\begin{align*}
\|\mathbf{u}_{h,\Delta t} - \mathbf{v}_{h,\Delta t}\|_{-\frac{1}{2},\frac{1}{2},\sigma, \ast}^2 &\lesssim \|\mathbf{u}-\mathbf{v}_{h,\Delta t}\|_{\frac{1}{2},\frac{1}{2},\sigma, \ast}\|\mathbf{u}_{h,\Delta t}- \mathbf{v}_{h,\Delta t}\|_{-\frac{1}{2},\frac{1}{2},\sigma, \ast}\\
& \quad +\|\tilde{\lambda}_{H,\Delta T} - \lambda\|_{\frac{1}{2},-\frac{1}{2},\sigma} \| \mathbf{u}_{h,\Delta t}-  \mathbf{v}_{h,\Delta t}\|_{-\frac{1}{2},\frac{1}{2},\sigma, \ast} \\ & \quad + \|\tilde{\lambda}_{H,\Delta T}-{\lambda}_{H,\Delta T}\|_{\frac{1}{2},-\frac{1}{2},\sigma} \| \mathbf{u}_{h,\Delta t}-  \mathbf{v}_{h,\Delta t}\|_{-\frac{1}{2},\frac{1}{2},\sigma, \ast}\ . 
\end{align*}
Dividing by $\|\mathbf{u}_{h,\Delta t} - \mathbf{v}_{h,\Delta t}\|_{-\frac{1}{2},\frac{1}{2},\sigma, \ast}$, we find
\begin{align*}
\|\mathbf{u}_{h,\Delta t} - \mathbf{v}_{h,\Delta t}\|_{-\frac{1}{2},\frac{1}{2},\sigma, \ast} &\lesssim_\sigma \|\mathbf{u}-\mathbf{v}_{h,\Delta t}\|_{\frac{1}{2},\frac{1}{2},\sigma, \ast}
+\|\tilde{\lambda}_{H,\Delta T} - \lambda\|_{\frac{1}{2},-\frac{1}{2},\sigma}  + \|\tilde{\lambda}_{H,\Delta T} - {\lambda}_{H,\Delta T}\|_{\frac{1}{2},-\frac{1}{2},\sigma}\ . 
\end{align*}
The triangle inequality now shows
\begin{align*}
\|\mathbf{u}-\mathbf{u}_{h,\Delta t}\|_{-\frac{1}{2},\frac{1}{2},\sigma, \ast} &\lesssim_\sigma \|\mathbf{u}-\mathbf{v}_{h,\Delta t}\|_{\frac{1}{2},\frac{1}{2},\sigma, \ast}
+\|\tilde{\lambda}_{H,\Delta T} - \lambda\|_{\frac{1}{2},-\frac{1}{2},\sigma} +\|\tilde{\lambda}_{H,\Delta T} - {\lambda}_{H,\Delta T}\|_{\frac{1}{2},-\frac{1}{2},\sigma} \,, 
\end{align*}
and estimate \eqref{est_2} follows.
\end{proof}

\begin{remark}
 { Note that the assumption $(\textbf{u},{\lambda}) \in H^{1/2}_\sigma(\mathbb{R}^+,\tilde{H}^{1/2}(\Gamma_\Sigma))^d\times {\widetilde{M}}^+$ in Theorem \ref{apriori} allows jump discontinuities in time of the velocity $\dot{\mathbf{u}} \in H^{-1/2}_\sigma(\mathbb{R}^+,\tilde{H}^{1/2}(\Gamma_\Sigma))^d$, as they occur in the case of an impact.}\\
 {The existence of a unique solution $(\textbf{u},{\lambda}) \in H^{1/2}_\sigma(\mathbb{R}^+,\tilde{H}^{1/2}(\Gamma_\Sigma))^d \times {\widetilde{M}}^+$ and the coercivity of $\mathcal{S}$ are known, for example, in the case of the analogous Signorini contact problem for the wave equation in simple geometries for sufficiently smooth right hand side $\textbf{f} \in H^{3/2}_\sigma\left(\mathbb{R}^+,H^{-1/2}(\Gamma_\Sigma)\right)^d$ and $g = 0$, see p.~450 in \cite{cooper} or (in a slightly larger function space for more general data) in \cite{lebeau}. }
\end{remark}

\begin{remark}\label{convremark}
{If the solution $(\mathbf{u}, \lambda)$ is more regular, we obtain convergence rates which are optimal in $h$, but suboptimal in $\Delta t$. Specifically, let us illustrate the results  when the time step $\Delta t$ is sufficiently small, $H=Ch$, $\Delta T = C\Delta t$ and in the spatial variables $\mathbf{u}$ is of regularity $H^{s}(\Gamma_\Sigma)$, $\lambda$ of regularity $H^{s-1}(\Gamma_C)$, where $s \in (\frac{1}{2}, \frac{3}{2}]$. Then there exist positive constants $C_\sigma'$ and $C_\sigma''$ such that
$$\|\mathbf{u}-\mathbf{u}_{h,\Delta t}\|_{-\frac{1}{2},\frac{1}{2},\sigma, \ast} \leq C_\sigma' h^{s-\frac{1}{2}}$$
and $$\| \lambda -\lambda_{H,\Delta T} \|_{0,-\frac{1}{2}, \sigma} \leq C_\sigma'' h^{s-\frac{1}{2}}.$$
Detailed, but suboptimal convergence rates for general  $\Delta t, \Delta T, h, H$ can be derived from the approximation results in \cite{hp,china}.}
\end{remark}

\section{Algorithmic details}\label{sec:psh}
In this section, to keep the notation as simple as possible and because the numerical results reported in Section \ref{sec;numres} are related to two-dimensional elastodynamics, we will fix $d=2$. 

\subsection{Implementation of the Poincar\'{e}-Steklov operator}
\label{subPS}
The numerical solution of the mixed formulation  \eqref{eq:MixedProblemh} of the contact problem will require the solution of the weak equation
\eqref{eq:WeakMixedVarEqh}, involving the discretized operator $\mathcal{S}_{h,\Delta t}$, at every iteration of the Uzawa algorithm, as described in Algorithm \ref{alg1}. Here, in particular, we are interested in the algorithmic details related to the numerical solution of such an equation.\\
 For the implementation of  the Poincar\'{e}-Steklov operator,
both the symmetric and the non-symmetric formulations \eqref{definition_S}, \eqref{definition_S_symm}, in terms of the layer potentials introduced in Section \ref{sec:Boundaryintegralformulation}, will be used. While they are equivalent for the continuous problem, the resulting discretizations will lead to slightly different numerical results, as we will comment in Section \ref{sec;numres}. 
Moreover, for the stability of the time-step scheme, we consider an energetic weak formulation which involves time derivatives of test functions.\\
In general, for a given right-hand side $\widetilde{\textbf{f}}\in H^{1/2}\left([0,T],H^{-1/2}(\Gamma_\Sigma)\right)^2$ we have to solve a weak problem of the form:\\

\textit{find} $\textbf{u}_{h,\Delta t} \in (X^0_{h,\Gamma_\Sigma}\otimes V^0_{\Delta t})^2$ \textit{such that} 
\begin{equation}\label{eq:Scontinuous}
\langle \mathcal{S} _{h,\Delta t}\textbf{u}_{h,\Delta t},\dot{\textbf{v}}_{h,\Delta t}\rangle_{0,\Gamma_\Sigma,(0,T]} = \langle \widetilde{\textbf{f}},\dot{\textbf{v}}_{h,\Delta t} \rangle_{0,\Gamma_\Sigma,(0,T]}, \quad \quad \forall\,\textbf{v}_{h,\Delta t} \in (X^0_{h,\Gamma_\Sigma}\otimes V^0_{\Delta t})^2.
\end{equation}

We first consider the definition of the Poincar\'{e}-Steklov operator in its symmetric form \eqref{definition_S_symm}: because of its dependence on the inverse operator ${\cal V}^{-1}$, to handle the equation at hand we need to define the dummy variable $\pmb{\psi}_{h,\Delta t}:=\mathcal{V}^{-1}\left(\mathcal{K}+\frac{1}{2}\right)\textbf{u}_{h,\Delta t}$ , belonging to the space $(X^{-1}_{h,\Gamma}\otimes V^{-1}_{\Delta t})^2$.

This translates in solving at every step of the Uzawa algorithm the following system of equations: \\ 

\textit{find} $\textbf{u}_{h,\Delta t}\in (X^0_{h,\Gamma_\Sigma}\otimes V^0_{\Delta t})^2$ \textit{and} $\pmb{\psi}_{h,\Delta t}\in (X^{-1}_{h,\Gamma}\otimes V^{-1}_{\Delta t})^2$ \textit{such that}
\begin{equation}\label{energetic_formh}
\begin{array}{l}
\left\langle \mathcal{V}\pmb{\psi}_{h,\Delta t},\dot{\pmb{\eta}}_{h,\Delta t}\right\rangle_{0,\Gamma,(0,T]} -\left\langle\left(\mathcal{K}+\frac{1}{2}\right)\textbf{u}_{h,\Delta t},\dot{\pmb{\eta}}_{h,\Delta t}\right\rangle_{0,\Gamma,(0,T]}=0, \quad\quad\quad\quad\quad\quad\quad\forall\, \pmb{\eta}_{h,\Delta t}\in(X^{-1}_{h,\Gamma}\otimes V^{-1}_{\Delta t})^2\\[4pt]
\left\langle\left(\mathcal{K}^*+\frac{1}{2}\right)\pmb{\psi}_{h,\Delta t}, \dot{\textbf{v}}_{h,\Delta t}\right\rangle_{0,\Gamma_\Sigma,(0,T]}-\left\langle \mathcal{W}\textbf{u}_{h,\Delta t}, \dot{\textbf{v}}_{h,\Delta t}\right\rangle_{0,\Gamma_\Sigma,(0,T]}=\left\langle\widetilde{\textbf{f}},\dot{\textbf{v}}_{h,\Delta t}\right\rangle_{0,\Gamma_\Sigma,(0,T]}, \forall\,\textbf{v}_{h,\Delta t}\in (X^0_{h,\Gamma_\Sigma}\otimes V^0_{\Delta t})^2.
\end{array}
\end{equation}

Taking into account the definition of the discrete polynomial spaces \eqref{time_set}, \eqref{pol_space_4_discontinuous_f} and \eqref{pol_space_4_continuous_f} , let us now consider the sets $\left\lbrace w^{(\pmb{\psi})}_m \right\rbrace_{m=1}^{M^{(\pmb{\psi})}_h}$ and $\left\lbrace w^{(\textbf{u})}_m \right\rbrace_{m=1}^{M^{(\textbf{u})}_h}$: the first contains the piece-wise linear basis functions of $X^{-1}_{h,\Gamma}$, while the second is composed by the piece-wise linear basis functions of $X^{0}_{h,\Gamma_\Sigma}$. The algebraic reformulation of \eqref{energetic_formh} can be then elaborated choosing the approximate components of the unknowns, in the spaces $X^{-1}_{h,\Gamma}\otimes V^{-1}_{\Delta t}$ and $X^0_{h,\Gamma_\Sigma}\otimes V^0_{\Delta t}$, of the following form:
\begin{equation}\label{approx_of_psi_and_u}
\psi_{i,h,\Delta t}(\textbf{x},t)=\sum_{\ell=0}^{N_{\Delta t}-1}\sum_{m=1}^{M^{(\pmb{\psi})}_h} \psi_{i,\ell,m}w^{(\pmb{\psi})}_m(\textbf{x})v_\ell(t),\:\:u_{i,h,\Delta t}(\textbf{x},t)=\sum_{\ell=0}^{N_{\Delta t}-1}\sum_{m=1}^{M^{(\textbf{u})}_h}u_{i,\ell,m}w^{(\textbf{u})}_m(\textbf{x})r_\ell(t),\:\: i=1,2,
\end{equation}
where the time basis $v_k$ and $r_k$ are defined as
$$v_\ell(t):=H[t-t_\ell]-H[t-t_{\ell+1}], \quad r_\ell(t):=H[t-t_\ell]\frac{t-t_\ell}{\Delta t}-H[t-t_{\ell+1}]\frac{t-t_{\ell+1}}{\Delta t}.$$

These choices for the approximation of $\textbf{u}_{h,\Delta t},\pmb{\psi}_{h,\Delta t}$ leads to the algebraic reformulation of \eqref{energetic_formh} as a linear system $\mathbb{S}\textbf{X}=\widetilde{\textbf{F}}$, having the form 
\begin{equation}\label{system}
\left(\begin{array}{cccc}
\mathbb{S}^{(0)} & \textbf{0} & \cdots & \textbf{0}\\
\mathbb{S}^{(1)} & \mathbb{S}^{(0)} & \cdots & \textbf{0}\\
\vdots & \vdots & \ddots & \vdots\\
\mathbb{S}^{(N_{\Delta t}-1)} & \mathbb{S}^{(N_{\Delta t}-2)} & \cdots & \mathbb{S}^{(0)}
\end{array}\right)
\left(
\begin{array}{c}
\textbf{X}_{(0)}\\
\textbf{X}_{(1)}\\
\vdots\\
\textbf{X}_{(N_{\Delta t}-1)}
\end{array}
\right)=
\left(
\begin{array}{c}
\widetilde{\textbf{F}}_{(0)}\\
\widetilde{\textbf{F}}_{(1)}\\
\vdots\\
\widetilde{\textbf{F}}_{(N_{\Delta t}-1)}
\end{array}
\right),
\end{equation}
where, for all time indices $\ell=0,...,N_{\Delta t}-1$, the unknown vectors are structured as
$$\textbf{X}_{(\ell)}=(\pmb{\Psi}_{(\ell)},\textbf{U}_{(\ell)})=\left(\psi_{1,\ell,1},\ldots,\psi_{1,\ell,M_h^{(\pmb{\psi})}},\psi_{2,\ell,1},\ldots,\psi_{2,\ell,M_h^{(\pmb{\psi})}},u_{1,\ell,1},\ldots,u_{1,\ell,M_h^{(\textbf{u})}},u_{2,\ell,1},\ldots,u_{2,\ell,M_h^{(\textbf{u})}}\right)^\top\,$$
and the right-hand side as
$$
\widetilde{\textbf{F}}_{(\ell)}=\left(0,\cdots,0,\,0,\cdots,0,\,\widetilde{F}_{1,\ell,1},\ldots,\widetilde{F}_{1,\ell,M_h^{(\textbf{u})}},\widetilde{F}_{2,\ell,1},\ldots,\widetilde{F}_{2,\ell,M_h^{(\textbf{u})}}\right)^\top 
\in \mathbb{R}^{2(M_h^{(\pmb{\psi})}+M_h^{(\textbf{u})})}\,,$$
with $\widetilde{F}_{i,\ell,m}:=\left\langle\widetilde{f}_i,w^{(\textbf{u})}_m\dot{r}_\ell\right\rangle_{0,\Gamma_\Sigma,(0,T]}$.\\
The matrix $\mathbb{S}$ shows a lower triangular Toeplitz structure, so that the system \eqref{system} can be solved by backsubstitution. This leads to a  \textit{marching-on-in-time} time stepping scheme, in which the solution of a system of equations involving the first time block $\mathbb{S}^{(0)}$ is required in every time step.

The structure of the linear system \eqref{system} and the expression of the matrix entries as double spatial integrals are due to the analytical time integration taken into account in formula \eqref{energetic_formh}. Let us note that the structure of a generic block $\mathbb{S}^{(\ell)}$ has dimension $2(M_h^{(\pmb{\psi})}+M_h^{(\pmb{\textbf{u}})})\times 2(M_h^{(\pmb{\psi})}+M_h^{(\pmb{\textbf{u}})})$ and assumes the form:
$$\mathbb{S}^{(0)}=\left(\begin{array}{cc}
\mathbb{V}^{(0)} & -(\mathbb{K}^{(0)}+1/2\:\mathbb{M})\\[4pt]
({\mathbb{K}}^{(0)}+1/2\:\mathbb{M})^\top & -\mathbb{W}^{(0)}
\end{array}\right),
\quad
\mathbb{S}^{(\ell)}=\left(\begin{array}{cc}
\mathbb{V}^{(\ell)} & -\mathbb{K}^{(\ell)}\\[4pt]
{\mathbb{K}^{(\ell)}}^\top & \mathbb{W}^{(\ell)}
\end{array}\right),\:\ell=1,...,N_{\Delta t}-1.
$$
With obvious meaning of notation, the blocks $\mathbb{V}^{(\ell)}$, $\mathbb{K}^{(\ell)}$ and $\mathbb{W}^{(\ell)}$ are obtained from the discretization of the integral operators involved in \eqref{energetic_formh} and $\mathbb{M}$ is a mass matrix. For details about the numerical evaluation the entries of these blocks the reader is referred to \cite{thesis_Giulia, AimiJCAM}.\\

We briefly also recall the problem we need to solve in case of non-symmetric formulation of the Poincar\'{e}-Steklov operator \eqref{definition_S}.  With similar arguments used for the symmetric formulation, we obtain the following system of weak boundary equations:\\

\textit{find} $\textbf{u}_{h,\Delta t}\in (X^0_{h,\Gamma_\Sigma}\otimes V^0_{\Delta t})^2$ \textit{and} $\pmb{\psi}_{h,\Delta t}\in (X^{-1}_{h,\Gamma}\otimes V^{-1}_{\Delta t})^2$ \textit{such that}
\begin{equation}\label{energetic_form_non_symmh}
\begin{array}{l}
\left\langle \mathcal{V}\pmb{\psi}_{h,\Delta t},\dot{\pmb{\eta}}_{h,\Delta t}\right\rangle_{0,\Gamma,(0,T]} -\left\langle\left(\mathcal{K}+\frac{1}{2}\right)\textbf{u}_{h,\Delta t},\dot{\pmb{\eta}}_{h,\Delta t}\right\rangle_{0,\Gamma,(0,T]}=0, \quad \forall\,\pmb{\eta}_{h,\Delta t}\in(X^{-1}_{h,\Gamma}\otimes V^{-1}_{\Delta t})^2\\[4pt]
\left\langle\pmb{\psi}_{h,\Delta t}, \dot{\textbf{v}}_{h,\Delta t}\right\rangle_{0,\Gamma_\Sigma,(0,T]}=\left\langle\widetilde{\textbf{f}},\dot{\textbf{v}}_{h,\Delta t}\right\rangle_{0,\Gamma_\Sigma,(0,T]}, \quad \quad \quad \quad \quad \quad \quad \quad \forall\,\textbf{v}_{h,\Delta t}\in (X^0_{h,\Gamma_\Sigma}\otimes V^0_{\Delta t})^2.
\end{array}
\end{equation}

The discrete unknowns $\textbf{u}_{h,\Delta},\pmb{\psi}_{h,\Delta}$ is substituted again with \eqref{approx_of_psi_and_u}, leading to an algebraic form $\mathbb{S}\textbf{X}=\widetilde{\textbf{F}}$ of \eqref{energetic_form_non_symmh}, analogous to the system \eqref{system}, with the only difference that the blocks $\mathbb{S}^{(\ell)}$, again with dimension $2(M_h^{(\pmb{\psi})}+M_h^{(\pmb{\textbf{u}})})\times 2(M_h^{(\pmb{\psi})}+M_h^{(\pmb{\textbf{u}})}),$ are structured as
$$\mathbb{S}^{(0)}=\left(\begin{array}{cc}
\mathbb{V}^{(0)} & -(\mathbb{K}^{(0)}+1/2\:\mathbb{M})\\[4pt]
\mathbb{M}^\top & \textbf{0}
\end{array}\right),
\quad
\mathbb{S}^{(\ell)}=\left(\begin{array}{cc}
\mathbb{V}^{(\ell)} & -\mathbb{K}^{(\ell)}\\[4pt]
\textbf{0} & \textbf{0}
\end{array}\right),\:\ell=1,...,N_{\Delta t}-1.
$$

\subsection{Algebraic formulation  of the Uzawa algorithm}
For the numerical solution of contact problems, the right-hand side $\widetilde{\textbf{f}}$ in \eqref{energetic_formh} and \eqref{energetic_form_non_symmh} specifies depending on the Neumann datum $\textbf{f}$ and the discretized Lagrange multiplier $\lambda_{H,\Delta T}(\textbf{x},t)$ as in \eqref{eq:WeakMixedVarEqh}.\\
Hence, let us consider the vector ${\bf F}=\left({\bf F}_{(\ell)}\right)_{\ell=0}^{N_{\Delta t}-1}$ with
$$
\textbf{F}_{(\ell)}=\left(0,\cdots,0,\,0,\cdots,0,\,F_{1,\ell,1},\ldots,F_{1,\ell,M_h^{(\textbf{u})}},F_{2,\ell,1},\ldots,F_{2,\ell,M_h^{(\textbf{u})}}\right)^\top
\in \mathbb{R}^{2(M_h^{(\pmb{\psi})}+M_h^{(\textbf{u})})}$$
and $F_{i,\ell,m}:=\left\langle f_i,w^{(\textbf{u})}_m\dot{r}_\ell\right\rangle_{0,\Gamma_\Sigma,(0,T]}$.\\ 
For the discretized Lagrange multiplier related to the simulations presented in Section \ref{sec;numres},
in practice we have chosen $H=h$ and $\Delta T = \Delta t$ and we have extended it trivially in vector form, in such a way that its Cartesian components can be expressed as
\begin{equation}\label{approx_of_lambda}
\lambda_{i,h,\Delta t}(\textbf{x},t)=\sum_{\ell=0}^{N_{\Delta t}-1}\sum_{m=1}^{M^{(\pmb{\lambda})}_h}\lambda_{i,\ell,m}w^{(\pmb{\lambda})}_m(\textbf{x})v_\ell(t),\; i=1,2,
\end{equation}
with $\left\lbrace w^{(\pmb{\lambda})}_m \right\rbrace_{m=1}^{M^{(\pmb{\lambda})}_h}$ the piece-wise constant basis functions of $X^{-1}_{h,\Gamma_C}$. 
For each time step $\ell=0,...,N_{\Delta t}-1$, we introduce the vectors\\ 
$$\pmb{\Lambda}_{(\ell)}=\left(\lambda_{1,\ell,1},\ldots,\lambda_{1,\ell,M_h^{(\pmb{\lambda})}},\lambda_{2,\ell,1},\ldots,\lambda_{2,\ell,M_h^{(\pmb{\lambda})}}\right)^\top$$

which are collected in the vector $\pmb{\Lambda}=\left(\pmb{\Lambda}_{(\ell)}\right)_{\ell=0}^{N_{\Delta t}-1}\in \mathbb{R}^{N_{\Delta t}2M_h^{(\pmb{\lambda})}}$.
We denote by $\mathcal{J}_\perp$ the set of those indices $j$ of the vector $\pmb{\Lambda}$ which correspond to the (nontrivial) components of $\pmb{\lambda}_{h,\Delta t}$ normal to $\Gamma$, and by $\mathcal{J}_\parallel$ the set of the remaining indices, 
corresponding to the (trivial) components  of $\pmb{\lambda}_{h,\Delta t}$ tangential to $\Gamma$.
The Uzawa algorithm given in Algorithm \ref{alg1} then translates into the following algebraic procedure where the stopping criterion is specified: 
\begin{algorithm}[H]
\caption{(Algebraic formulation of Uzawa algorithm)}
\label{alg2}
\begin{algorithmic}
\STATE Fix $\rho>0$ and $\epsilon>0$.
\STATE $k=0$,  $\pmb{\Lambda}^{(0)}= \textbf{0}$ and $\pmb{\Lambda}^{(-1)}=\textbf{1}$
\WHILE{$\Vert\pmb{\Lambda}^{(k)}-\pmb{\Lambda}^{(k-1)} \Vert_2/\Vert\pmb{\Lambda}^{(k)} \Vert_2> \epsilon$}
\STATE \textbf{solve} $\quad$ $\mathbb{S}\textbf{X}^{(k)}=\textbf{F}+\mathbb{M}^*\pmb{\Lambda}^{(k)}$
\STATE \textbf{extract}$\quad$
$\textbf{U}^{(k)}$ from $\textbf{X}^{(k)}$
\STATE \textbf{compute} $\quad$ $\pmb{\Lambda}^{(k+1)}= \text{pr}_C  (\pmb{\Lambda}^{(k)}  - \rho \widetilde{\mathbb{M}}(\textbf{U}^{(k)}-\textbf{G}))$. 
\STATE $k \leftarrow k+1$
\ENDWHILE
\end{algorithmic}
\end{algorithm}
where $\text{pr}_C: \mathbb{R}^{N_{\Delta t}2M_h^{(\pmb{\lambda})}}\rightarrow \mathbb{R}^{N_{\Delta t}2M_h^{(\pmb{\lambda})}}$ is understood as the discretized version of the projector considered in Algorithm \ref{alg1}, acting on a 
vector ${\bf W}$ as 
\begin{equation}
(\text{pr}_C\ {\bf W})_j= \left\{
\begin{array}{lr}
\max \,\{W_j,0\}, & j \in \mathcal{J}_\perp\\
0,& j\in \mathcal{J}_\parallel
\end{array}\,.
\right.
\end{equation}
The matrix $\mathbb{M}^*\in \mathbb{R}^{N_{\Delta t}\,2(M_h^{(\pmb{\psi})}+M_h^{(\textbf{u})})\times N_{\Delta t}2M_h^{(\pmb{\lambda})}}$ represents the projection of discretized Lagrange multipliers on the discretized displacement space, trivially extended in such a way that the vector  $\mathbb{M}^*\pmb{\Lambda}^{(k)}$ matches the length of vector ${\bf F}$;
the vector $\textbf{G}$ contains the coefficients of the interpolant of $g$ in $X^0_{h,\Gamma_C}\otimes V^0_{\Delta t}$, suitably trivially extended to match the length of vector $\textbf{U}^{(k)}$, while 
$\widetilde{\mathbb{M}} \in \mathbb{R}^ {N_{\Delta t}2M_h^{(\pmb{\lambda})}\times N_{\Delta t}2M_h^{(\textbf{u})}}$ is the mass matrix representing the interplay between the finite dimensional spaces of discretized Lagrange multipliers and discretized displacements.

\section{Numerical results}
\label{sec;numres}

In the following we present numerical experiments which apply the proposed methods to the approximation of two-dimensional ($d=2$) dynamic Signorini contact problems. Different boundary geometries (slit, square boundary and circumference) and contact conditions are considered, illustrating the stability and efficiency of the approach.\\
Unless otherwise specified, the gap function in contact conditions \eqref{contactbc} and the mass density are respectively set to be $g=0$ and $\varrho=1$.\\

To study the convergence of the numerical solutions, we consider the error in energy norm. Given a solution vector $\textbf{X}^{(k)}$ to the linear system $\mathbb{S}\textbf{X}^{(k)}=\textbf{F}+\mathbb{M}^*\pmb{\Lambda}^{(k)}$ in Algorithm \ref{alg2},  once the stopping criterion in this algorithm is satisfied, we compute the square of the energy norm as $\textbf{X}^{(k)\top} \mathbb{S}\textbf{X}^{(k)}$. In the examples where the exact solution is unknown, 
 we compare these energies with an 
 extrapolated benchmark value. 

\subsection{Example 1: non-symmetric and  symmetric formulation of the Poincar\'{e}-Steklov operator}

In the time interval $[0,T]=[0,2]$, let us consider the square centred in the origin of the fixed orthogonal reference system and with sides of unitary length, i.e. $\Omega=[-0.5,0.5]^2$. We denote the square boundary by $\Gamma:=\partial \Omega=\Gamma_b\cup \Gamma_r\cup \Gamma_t\cup \Gamma_l$, with obvious meaning of notation for the bottom, right, top and left sides. The elastodynamics velocities are fixed as $c_S=1/\sqrt{2}$ and $c_P=1$: in this way the problem at hand results decoupled w.r.t. the horizontal and vertical directions \cite{Eringen1975}.

Here the partition of the boundary is trivial: $\Gamma=\Gamma_N$. Ipso facto, we are dealing with a dynamics of interior elastic wave propagation with Neumann boundary conditions, here determined by the vertical datum $\textbf{f}=(0,f_2)$, with

$$
f_2(\textbf{x},t)=\left\{ \begin{array}{c l}
0, & \textbf{x}\in\Gamma_l\cup\Gamma_r\\
1, & \textbf{x}\in\Gamma_t\\
-2\,H(t-1), & \textbf{x}\in\Gamma_b
\end{array}\right. .
$$
Since $\Gamma_C=\emptyset$ the implementation of Algorithm \ref{alg2} is trivial, converging in fact in one iteration. For this assigned Neumann datum, the displacement solution of \eqref{navierlame} is analytically known \cite{Eringen1975} and reads:
\begin{equation}\label{analytical_ex_1}
 \begin{array}{l}
 u_1 (\textbf{x},t)=0,\\
 u_2 (\textbf{x},t)= H(c_P\, t-x_2+0.5)\,(c_P\, t-x_2+0.5)-H(c_P\, t-x_2-1.5)\,(c_P\, t-x_2-1.5)\,.
\end{array}
\end{equation}

The related discrete energetic weak boundary integral problem \eqref{energetic_formh} or \eqref{energetic_form_non_symmh} is solved setting uniform meshes in space and time, generated by fixing space and time steps of equal value: $h=\Delta t$.\\

Making use of the non-symmetric formulation of the Poincar\'{e}-Steklov operator \eqref{definition_S} and setting $h=\Delta t=0.05$, we obtain a good approximation $u_{1,h,\Delta t}$ of the trivial horizontal displacement. The approximation of the vertical displacement, in Figure \ref{Figure1}(a), is instead more intriguing. The comparison between the numerical solution $u_{2,h,\Delta t}$ and the corresponding analytical one in the midpoint of the top, left (right) and bottom sides of the square boundary is optimal, given that the respective three lines are overlapped. In Figure \ref{Figure1}(b) the surface of $u_{2,h,\Delta t}$ on the entire space-time domain is reported.\\
\begin{figure}[h!]
\centering
\subfloat[]{\includegraphics[scale=0.54]{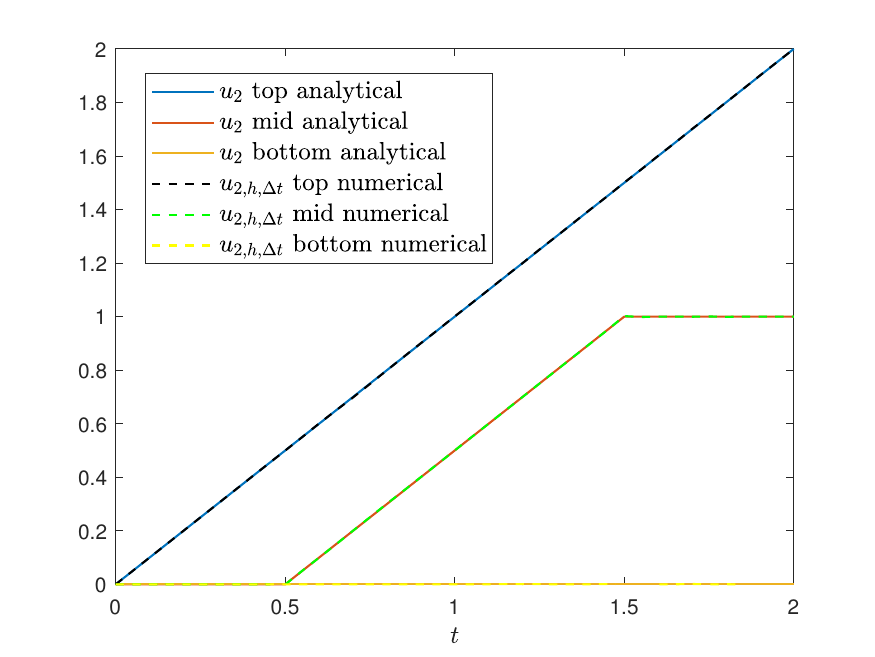}}
\subfloat[]{\includegraphics[scale=0.54]{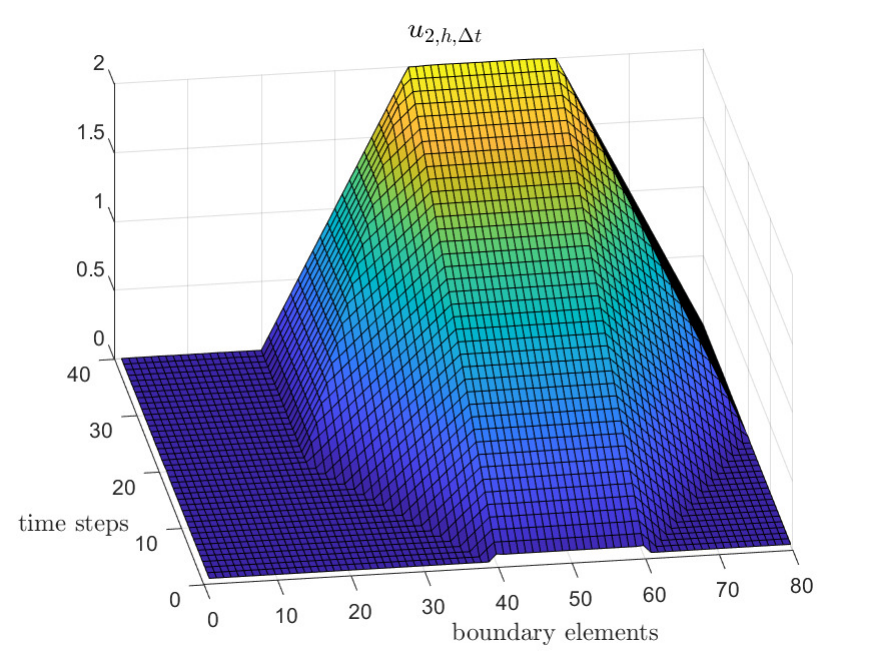}}
\caption{Comparison between $u_{2,h,\Delta t}$ and the corresponding analytical solution (a) and $u_{2,h,\Delta t}$ surface on the entire space-time domain (b).}
\label{Figure1}
\end{figure}

Concerning Figure \ref{Figure2}, results about the error analysis are collected, considering a resolution of the problem by both non-symmetric and symmetric couplings (see \eqref{definition_S} and \eqref{definition_S_symm}). This analysis is partially summarized in Figure \ref{Figure2}(a), which shows the space-time error in $L^2((0,T)\times \Gamma)$-norm between the discrete components of the solution $u_{1,h,\Delta t},u_{2,h,\Delta t}$  and the analytical ones reported in \eqref{analytical_ex_1}. For both types of coupling, the errors behave as $\mathcal{O}(h^{1.5})$.\\
Considering the same levels of space-time discretization, Figure \ref{Figure2}(b) collects the squared energy error committed with both non-symmetric and symmetric couplings of the problem: this error is characterized by a linear linear slope, behaving in fact as $\mathcal{O}(h)$.
\begin{figure}[h!]
\centering
\subfloat[]{\includegraphics[width=1\textwidth]{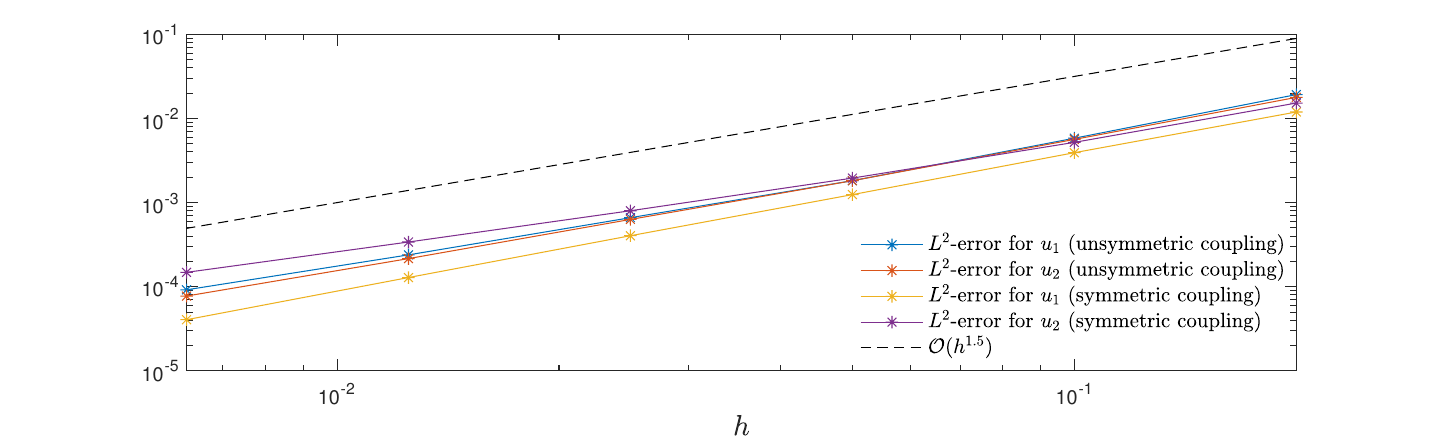}}
\vskip 0.05cm
\subfloat[]{\includegraphics[width=1\textwidth]{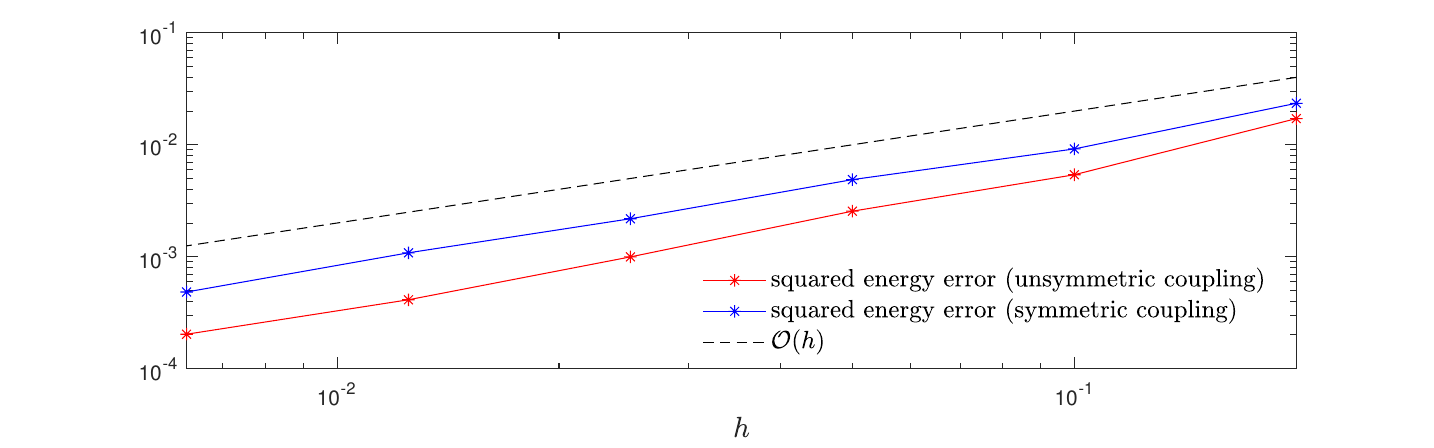}}
\caption{$L^2((0,T)\times \Gamma)$ space-time error committed w.r.t to the analytical solution \eqref{analytical_ex_1} (a) and squared error in energetic norm (b).}
\label{Figure2}
\end{figure}

\subsection{Example 2: dynamic contact problem on a segment}
Here we consider the slit $\Gamma=\left\lbrace (x_1,0)\:\vert\:x_1\in[-0.5,0.5]\right\rbrace$ with a contact region identified by the sub-interval $\Gamma_C=\left\lbrace (x_1,0)\:\vert\:x_1\in[-0.2,0.2]\right\rbrace$. For this example we set $\Gamma_{\Sigma}=\Gamma_C$ and, since on $\Gamma_D=\Gamma \setminus \Gamma_{\Sigma}$ the datum $\textbf{u}=0$ is imposed, the slit results partially fixed on both sides and the dynamics purely depends on an assigned contact force $\textbf{f}$. The considered wave speeds are $c_S=1, c_P=2$ and the time interval of analysis is $[0,T]=[0,3]$.\\
Two different vertical contact forces \textbf{f} have been considered on $\Gamma_C$, fixing for both a null horizontal impulse $f_1=0$: 
$$
f_2(t)=\left\{ \begin{array}{l r}
-\sin(8\pi t) H(0.1875-t)+H(t-0.1875) & {\rm for\, Test\, 1}, \\
-\sin(16\pi t) H(0.21875-t)+H(t-0.21875) & {\rm for\, Test\, 2}. \\
\end{array}\right.
$$
Note that the slit $\Gamma$  considered in this Example does not quite satisfy the assumptions described in Section \ref{mechsetup}. It approximates the contact between an infinite upper half-space and a rigid lower half-space. By truncating the infinitely extended interface to a finite segment $\Gamma$, we incur geometric truncation errors in the numerical discretizations of $\mathcal{S}$.\\
The algebraic Uzawa Algorithm \ref{alg2} 
is used with a tolerance $\epsilon=10^{-5}$, while the parameter for the projector $\textrm{pr}_C$ has been chosen as $\rho=10^5$. The discretization space-time steps can assume different values but maintaining always the ratio $h/\Delta t=2$. The numerical results are compared against an extrapolated benchmark for the symmetric, respectively non-symmetric coupling. \\
Setting the space step $h=0.025$, in Figure \ref{Figure3} we report the corresponding vertical component of displacement $u_{2,h,\Delta t}$ in the midpoint of the contact region $\Gamma_C$, together with the vertical component of the contact force $\textbf{f}$ taken into account. This force, when positive, causes a not trivial displacement solution on $\Gamma_C$, since the contact region is pulled up by the contact force. Otherwise, when the vertical contact force is negative, the resultant vertical displacement decreases to $0$ and the contact region turns to its original flat form since the vertical displacement cannot become negative. In this Figure, $\lambda_{2,h,\Delta t}$ is also shown.\\
With reference to Test 1, having fixed the same discretization parameters as before, in Figure \ref{Figure3_1} the deformation of the contact region is represented through the corresponding approximated horizontal and vertical components of ${\bf u}_{h,\Delta t}$ at three time instants when the prescribed vertical force is positive. With reference to Test 2, Figure \ref{Figure3_3} shows the approximated $u_{1,h,\Delta t}$ and $u_{2,h,\Delta t}$ at three time instants at the beginning of the simulation. Looking in particular at the vertical component, at first the displacement remains trivial, then it starts assuming positive values, before decreasing, with more evidence towards the endpoints of the contact region, when the vertical contact force becomes negative again.\\ 
In Figure \ref{Figure3_2}, the approximation $\lambda_{2,h,\Delta t}$ is shown, both for Test 1 and Test 2, on the global space-time domain $\Gamma_C\times[0,3]$, having set $ h=0.025$. As one can see, in Figure \ref{Figure3_2} (a), related to Test 1, the non-trivial part of the surface corresponds to a nonzero  contact force for $t\in [0, 0.125]$ to compensate the negative forcing term $f_2(t)$ applied uniformly on $\Gamma_C$. The shape of $\lambda_{2,h,\Delta t}$ surface given by Test 2, Figure \ref{Figure3_2} (b), presents similar features in the initial phase, but, at the second negative pulsation of the contact force,  $\lambda_{2,h,\Delta t}$ is non-trivial just towards the endpoints of $\Gamma_C$, since in the middle of the segment the vertical displacement has a decreasing, but still positive value.\\ 
Let us remark that both ${\bf u}_{h,\Delta t}$ and $\lambda_{h,\Delta t}$ are retrieved from the final Uzawa iteration of Algorithm \ref{alg2}, once the stopping criterion is satisfied. 
\begin{figure}[h!]
\centering
\subfloat[]{\includegraphics[scale=0.54]{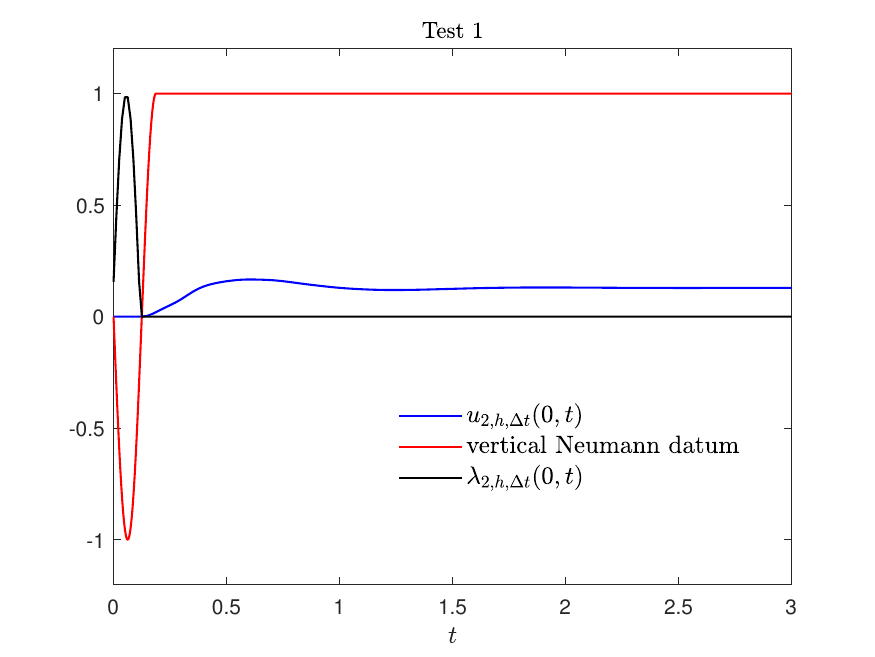}}
\subfloat[]{\includegraphics[scale=0.54]{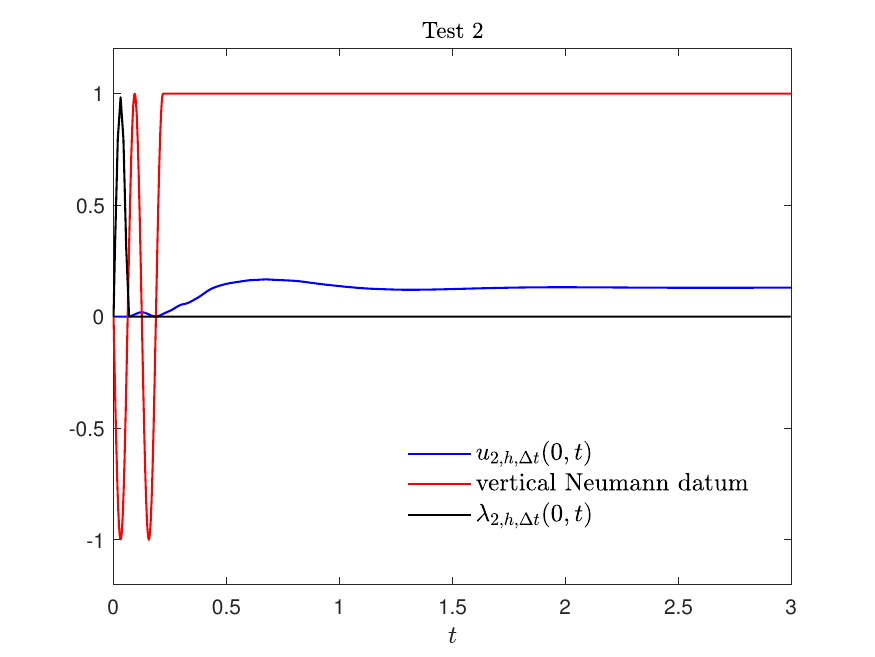}}
\caption{Time history of $u_{2,h,\Delta t}$ in the midpoint of $\Gamma_C$ for Test 1 (left), Test 2 (right) together with the corresponding vertical contact force and $\lambda_{2,h,\Delta t}$.}
\label{Figure3}
\end{figure}
\begin{figure}[h!]
\includegraphics[scale=0.55]{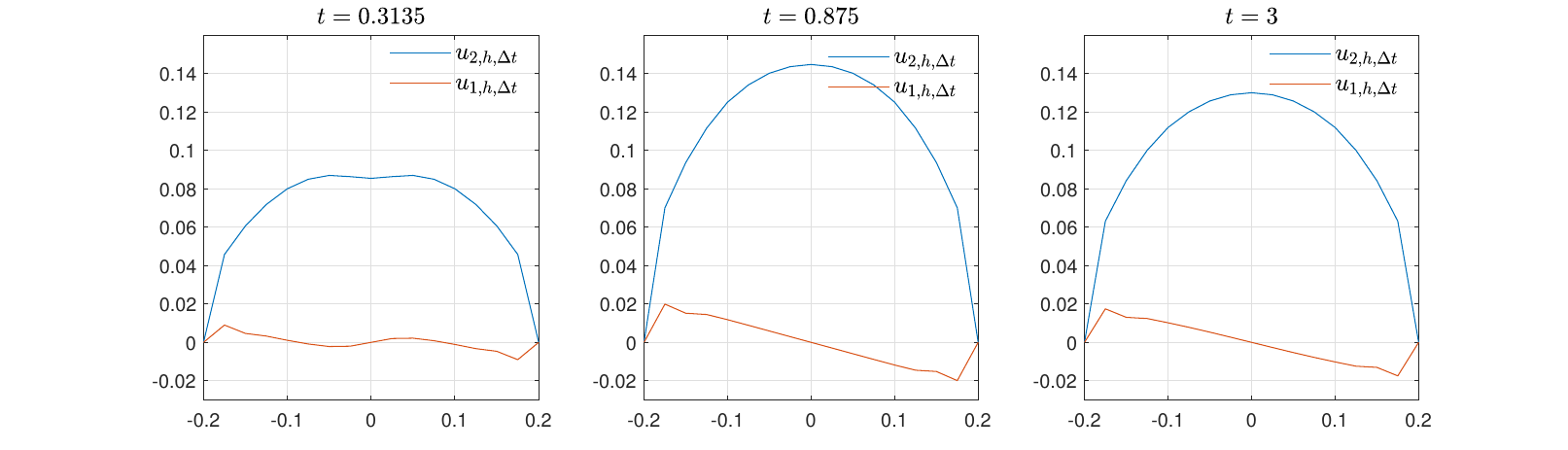}
\caption{Vertical and horizontal deformation of $\Gamma_C$ at three different time instants (Test 1).}
\label{Figure3_1}
\end{figure}
\begin{figure}[h!]
\includegraphics[scale=0.55]{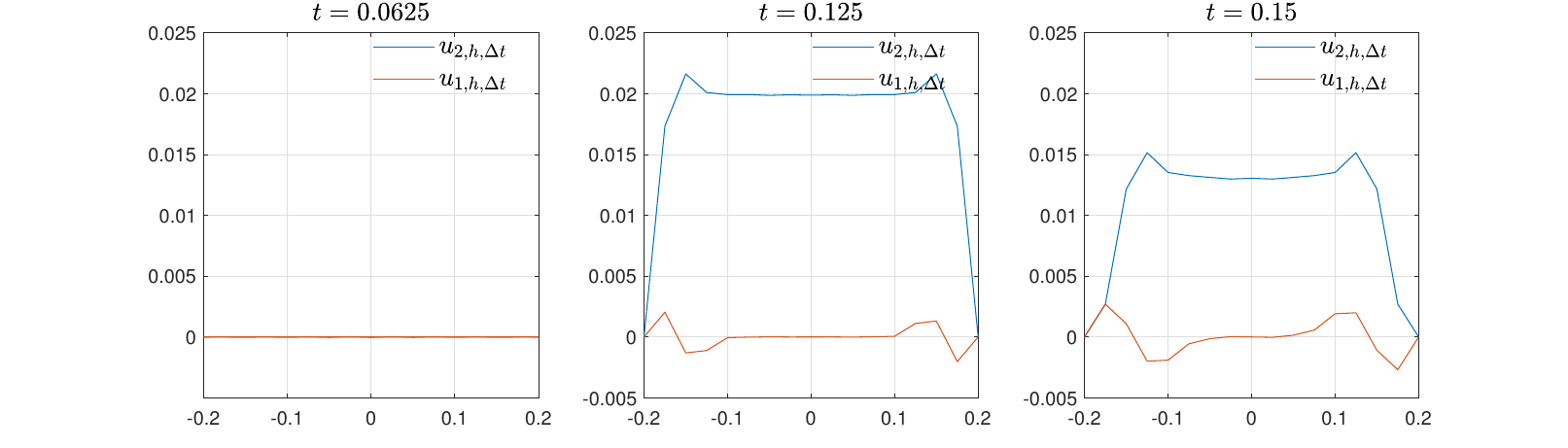}
\caption{Vertical and horizontal deformation of $\Gamma_C$ at three different time instants (Test 2).}
\label{Figure3_3}
\end{figure}
\begin{figure}[h!]
\centering
\subfloat[]{\includegraphics[scale=0.53]{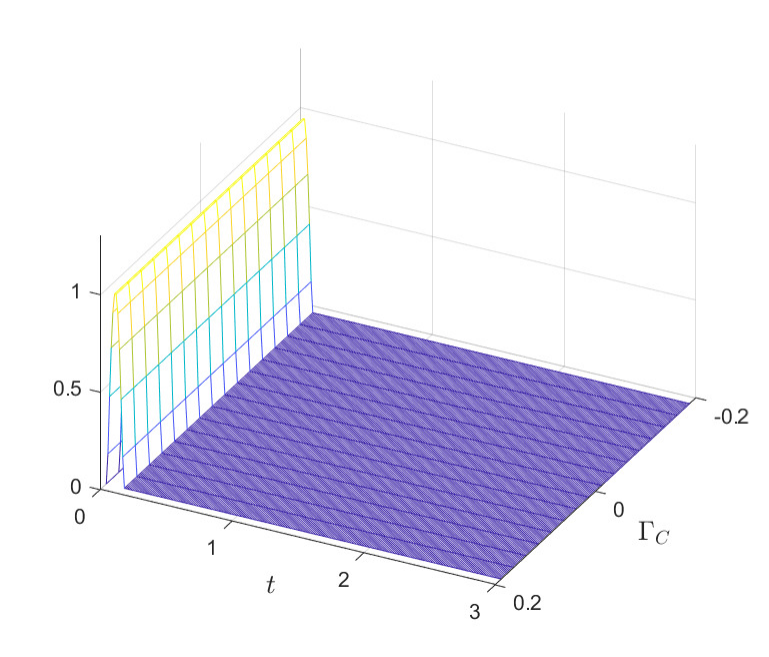}}
\:
\subfloat[]{\includegraphics[scale=0.53]{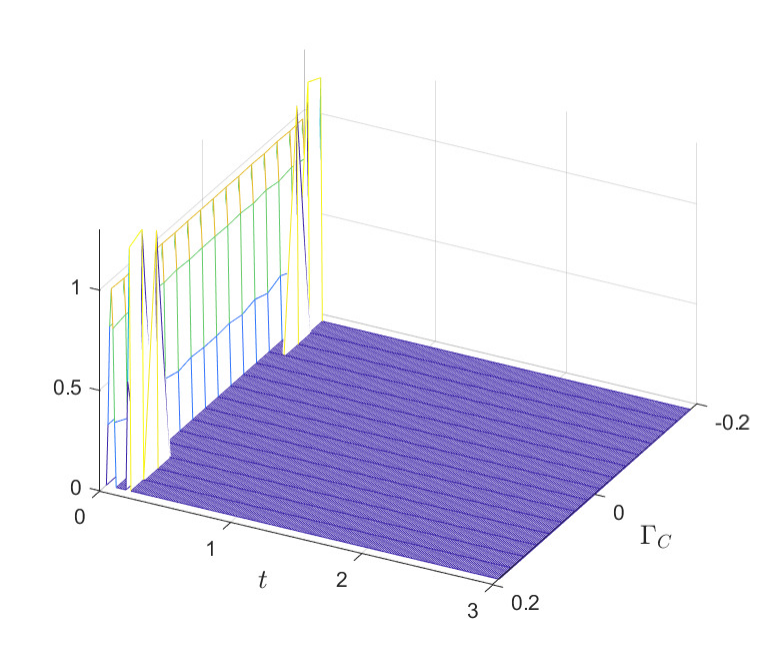}}
\caption{Surface of $\lambda_{2,h,\Delta t}$ for Test 1 (a) and Test 2 (b) on the entire space-time domain $\Gamma_C\times[0,3]$.}
\label{Figure3_2}
\end{figure}\\
The corresponding error analysis for this experiment is reported in Figure \ref{Figure4}: the squared error in energy norm is indicated distinguishing between Test 1 and Test 2 (Figure \ref{Figure4}(a)) and (Figure \ref{Figure4}(b)) and for a subsequent refinement of the $h$ step. For both experiments the straight line in the plots correspond to convergence of order $\mathcal{O}(h)$. {Following Remark \ref{convremark}, the convergence rate is compatible with a solution $\mathbf{u}$ of  regularity $H^{1-\varepsilon}(\Gamma_\Sigma)$ in the spatial variables.}
\begin{figure}[h!]
\centering
\subfloat{\includegraphics[scale=0.5]{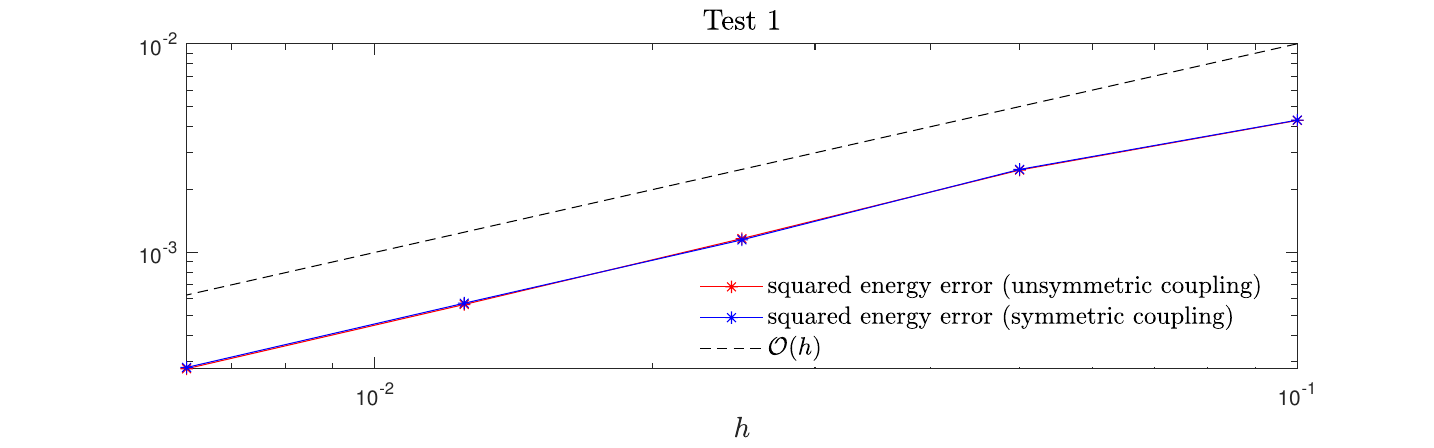}}
\vskip 0.05cm
\subfloat{\includegraphics[scale=0.5]{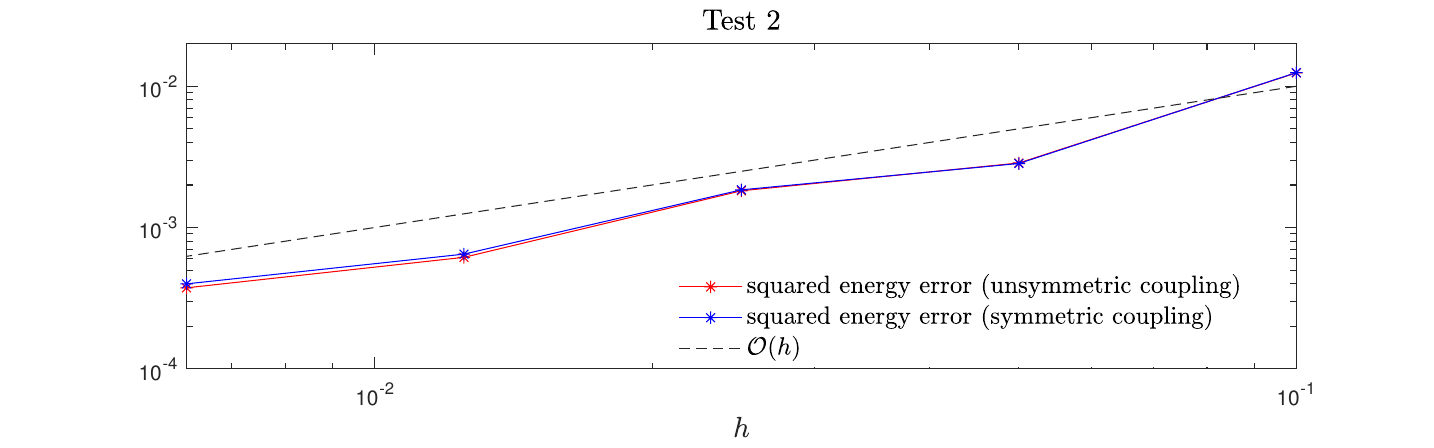}}
\caption{Squared error in energy norm related to the Example 2 (Test 1 top, Test 2 bottom).}
\label{Figure4}
\end{figure}

\subsection{Example 3: dynamic contact problem in non-flat geometry}

For this example, we take into account again the square $\Omega=[-0.5,0.5]^2$, with $\Gamma_N=\Gamma_t\cup \Gamma_r$ and the same physical and discretization parameters considered in the Example 1. Let us define the contact region as $\Gamma_C=\Gamma_b\cup \Gamma_l$. The dynamic will be investigated using two contact forces, always fixing $\textbf{f}=\textbf{0}$ on $\Gamma_C$, $f_2(\textbf{x},t)=-0.1H[t]$ on $\Gamma_t$, $f_2(\textbf{x},t)=0$ on $\Gamma_r$ and 
$$
f_1(\textbf{x},t)=\left\{ \begin{array}{c r}
0 \;{\rm on} \, \Gamma_t\cup \Gamma_r, & {\rm for\, Test\, 1} \\
-0.1H[t] \,{\rm on} \, \Gamma_r;\; 0\, {\rm on}\, \Gamma_t, & {\rm for\, Test\, 2} \\
\end{array}\right..
$$
Figure \ref{Figure5} perfectly shows how the two types of contact forces interact with the square: in Test 1 this is pushed down from the top side, while for Test 2 the square is pressed both from the top and from the left side.\\

The Uzawa algorithm has been set up with an exit test 
fixing $\epsilon=10^{-5}$, while the parameter for the projector has been fixed as $\rho=10^2,10^3,10^4$ for decreasing mesh parameters.\\
\begin{figure}[h!]
\centering
\subfloat[]{\includegraphics[scale=0.6]{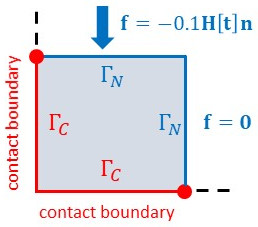}}
\quad\quad\quad\quad
\subfloat[]{\includegraphics[scale=0.6]{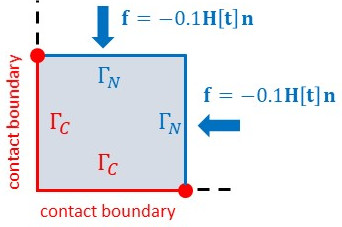}}
\caption{Schematic representation of the two types of force having role in Test 1 and Test 2, respectively (a) and (b), in Example 3.}
\label{Figure5}
\end{figure}

Using the non-symmetric formulation of $\mathcal{S}$ given in \eqref{definition_S}, having fixed $h=\Delta t=0.05$ we obtain the approximate displacements shown in Figure \ref{Figure6} related to Test 1 and in Figure \ref{Figure7} for Test 2.
\begin{figure}
\includegraphics[scale=0.53]{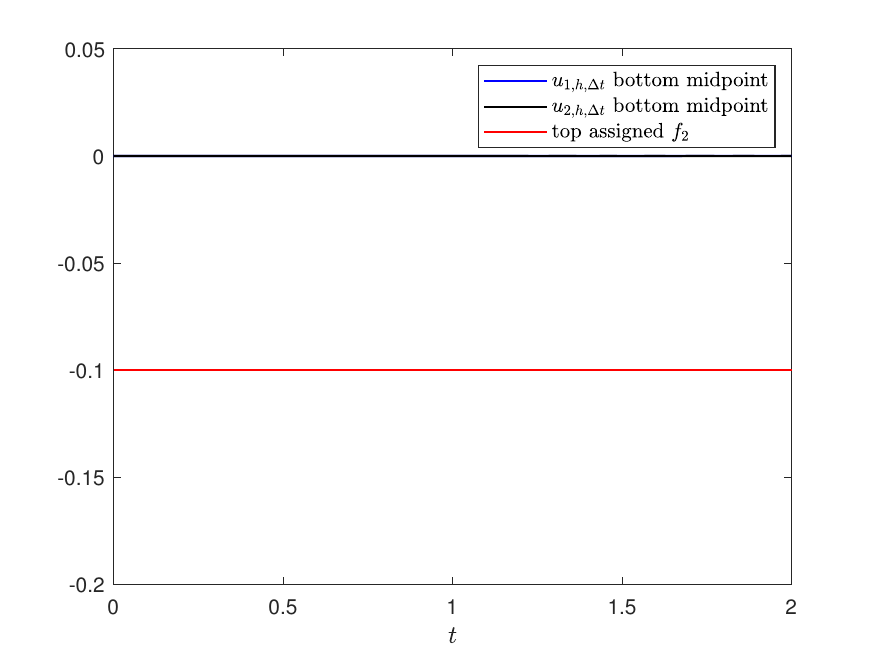}
\includegraphics[scale=0.53]{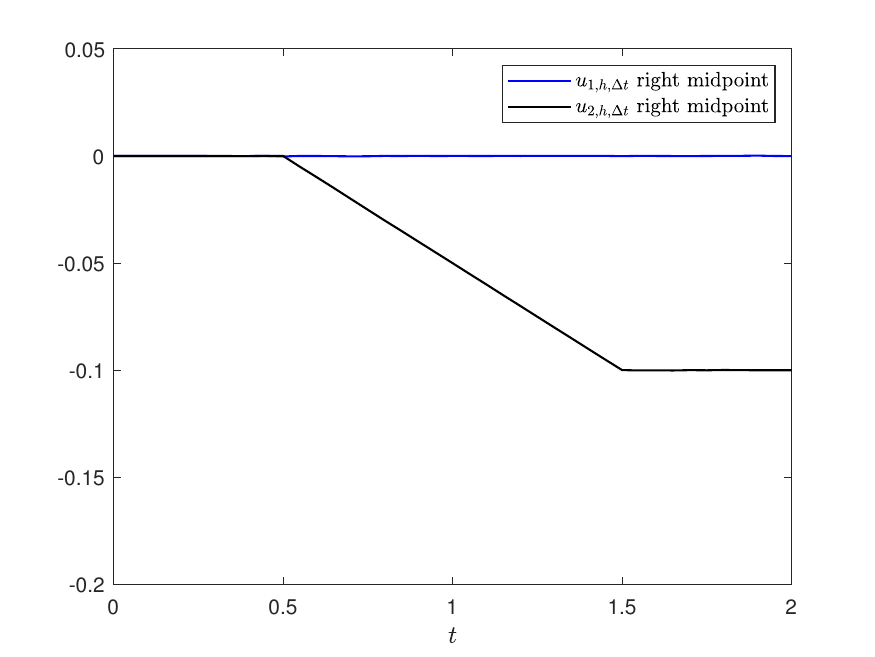}
\vskip 0.1cm
\includegraphics[scale=0.53]{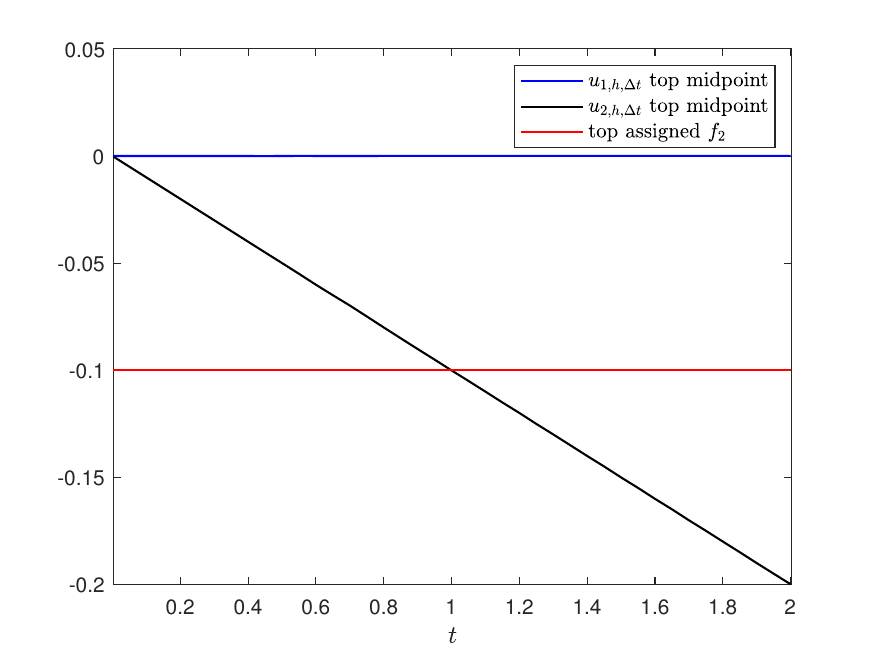}
\includegraphics[scale=0.53]{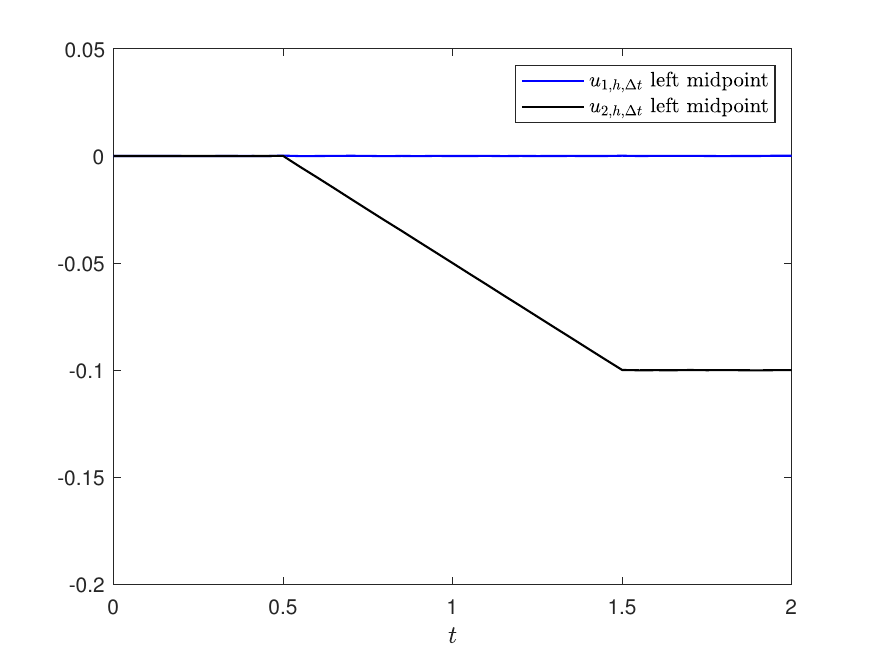}
\caption{Time history of $\textbf{u}_{h,\Delta t}$ components in the midpoint of $\Gamma$ sides for Test 1, together with the corresponding vertical Neumann datum.}
\label{Figure6}
\end{figure}
\begin{figure}
\includegraphics[scale=0.53]{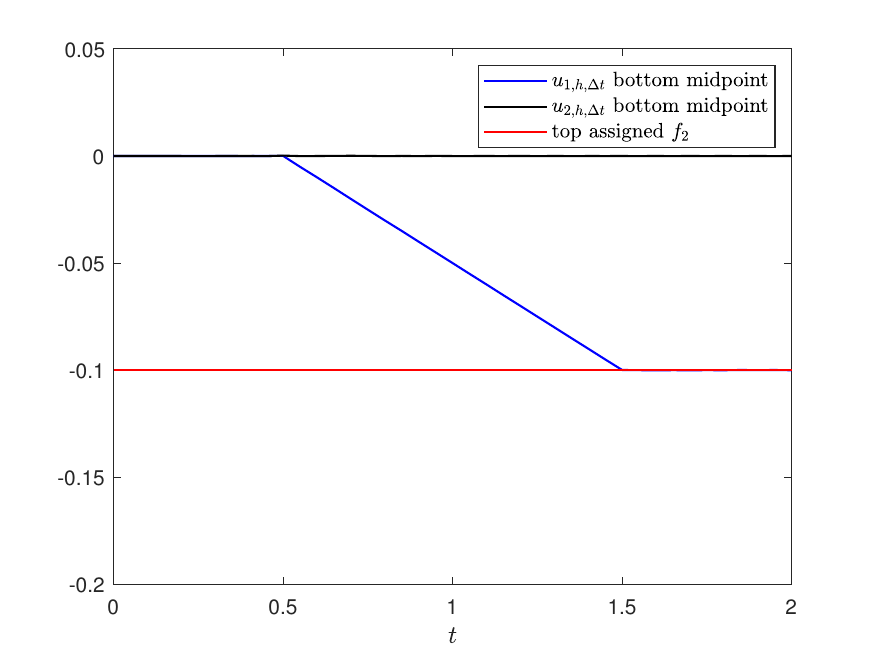}
\includegraphics[scale=0.53]{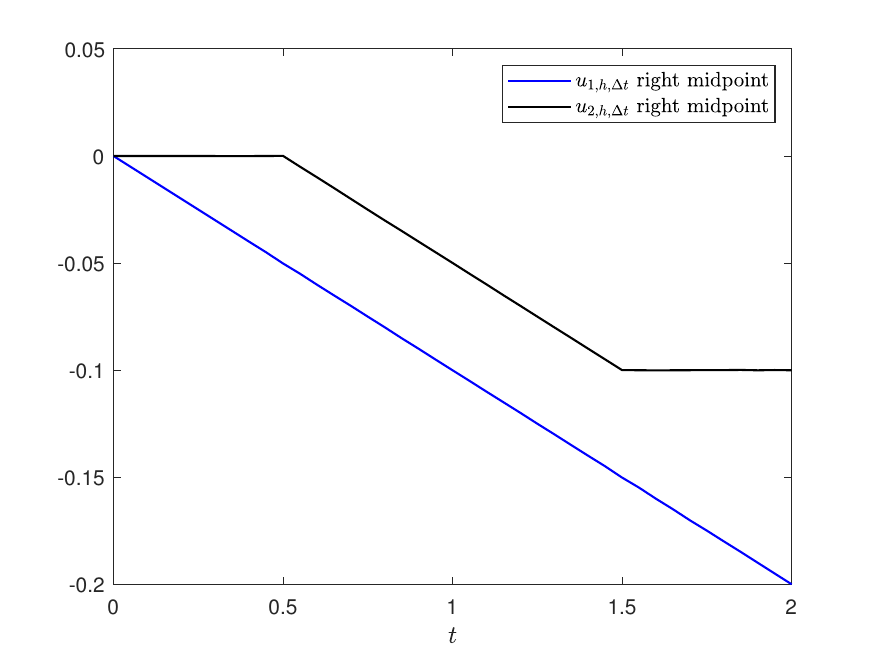}
\vskip 0.1cm
\includegraphics[scale=0.53]{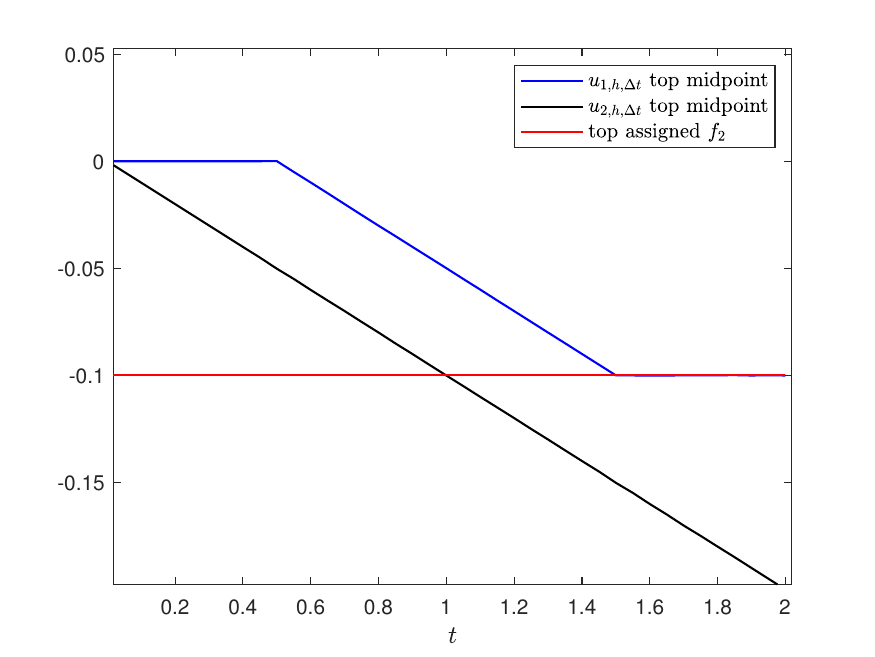}
\includegraphics[scale=0.53]{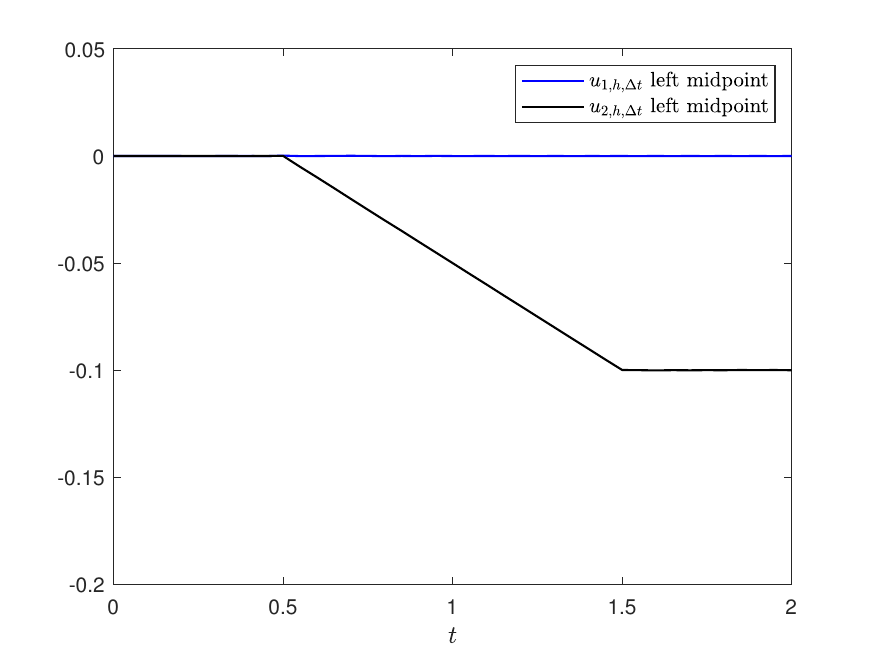}
\caption{Time history of $\textbf{u}_{h,\Delta t}$ components in the midpoint of $\Gamma$ sides for Test 2, together with the corresponding vertical Neumann datum.}
\label{Figure7}
\end{figure}
Figure \ref{Figure8} shows the squared energy error decay for $h$-refinements. For both Test 1 and Test 2, the considered representations \eqref{definition_S} and \eqref{definition_S_symm} of the Poincar\'{e}-Steklov operator lead to a linear slope for the squared error, parallel to the line $\mathcal{O}(h)$. {As in Example 2, the convergence rate is compatible with a solution $\mathbf{u}$ of  regularity $H^{1-\varepsilon}(\Gamma_\Sigma)$ in the spatial variables.}\\
\begin{figure}[h!]
\centering
\subfloat[]{\includegraphics[scale=0.7]{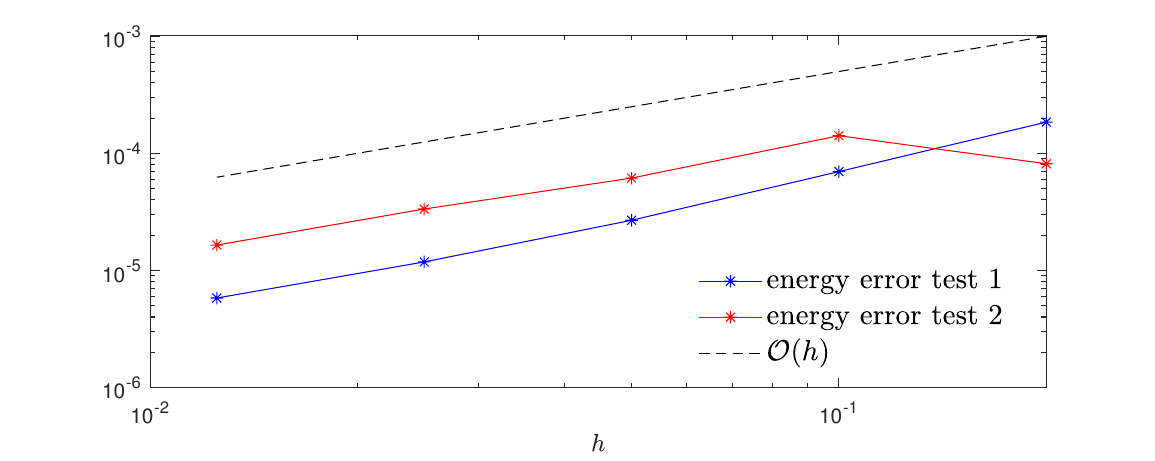}}
\vskip 0.05cm
\subfloat[]{\includegraphics[scale=0.7]{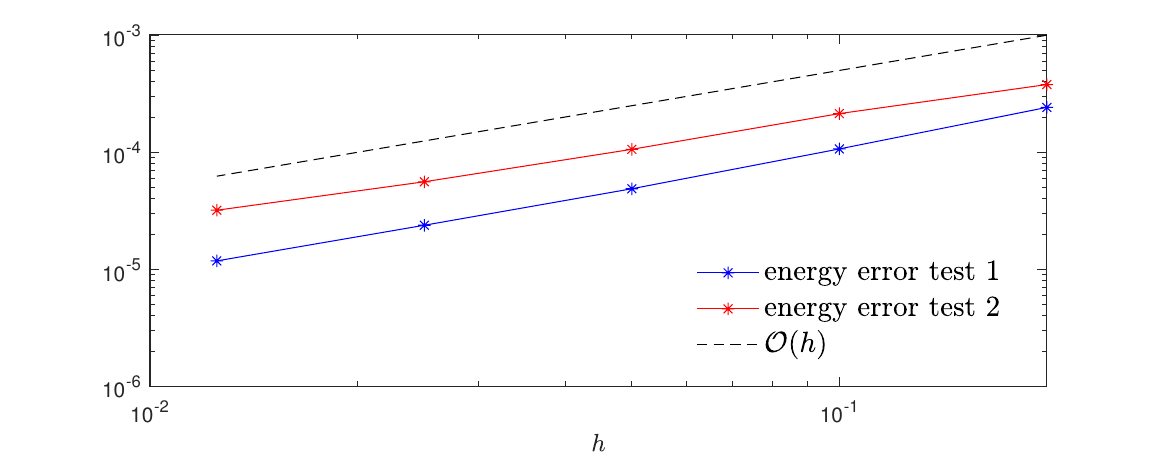}}
\caption{Squared error in energy norm related to the Example 3: figures (a) and (b) include errors obtained by non-symmetric and symmetric coupling of the problem.}
\label{Figure8}
\end{figure}

{
With the aim of studying the conservation of the energy after multiple bounces, let us consider the configuration depicted in Figure \ref{Figure5} (a). The elastic properties take the same values as above, but we change the applied Neumann vertical datum at the top side of the square boundary, namely
$$f_2({\bf x},t)=0.1\;\sum_{k=0}^8 (-1)^{k}\left(H[t-0.1\, k]-H[t-0.1\,(k+1)]\right),\quad {\bf x}\in \Gamma_t.$$
For fixed $h=\Delta t=0.0125$, Figure \ref{Figure12} shows the time history of the displacements components in the midpoints of the top (a) and bottom (b) sides of the square boundary, together with $f_2({\bf x},t)$ in the same points. In particular, in Figure \ref{Figure12} (b), after a delay due to the time for the elastic wave to reach the bottom of the square from the top side, multiple bounces against the contact region are highlighted. No spurious oscillations are observed even after multiple contacts. For the same discretization parameters, Figure \ref{Figure13} (a) shows the energy of the elastic system as a function of time. The approximately linear increase for times $<1$ and constant behavior for $t>1$ agrees with the depicted behavior of the solution and the applied forces. Figure \ref{Figure13} (b) zooms in around the stationary behavior. It shows that the energy is conserved up to a relative error of less than $10^{-4}$  for $t>1$. 
\begin{figure}[h!]
\centering
\subfloat[]{\includegraphics[scale=0.5]{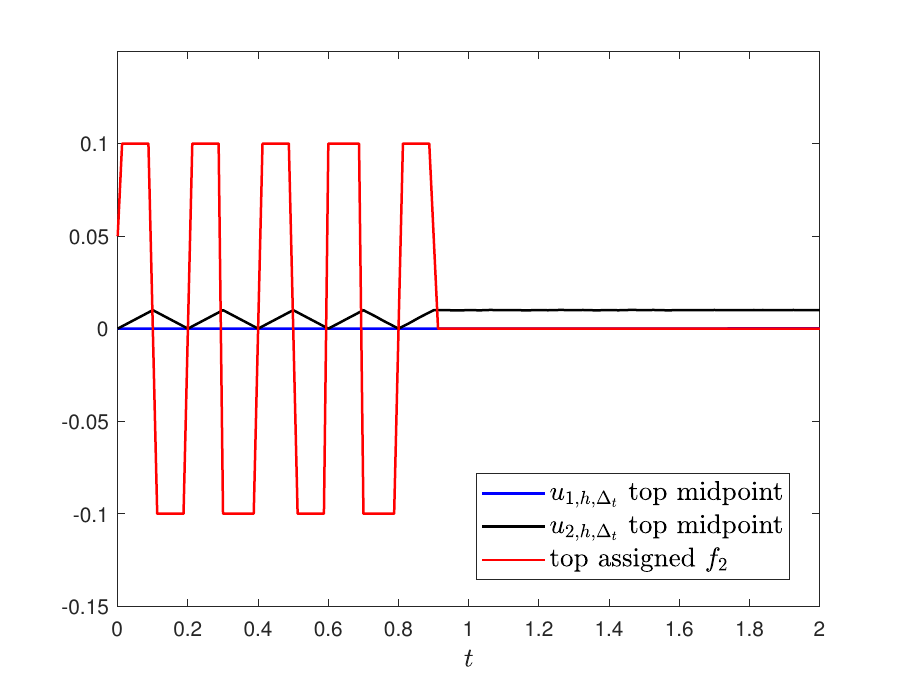}}
\:
\subfloat[]{\includegraphics[scale=0.5]{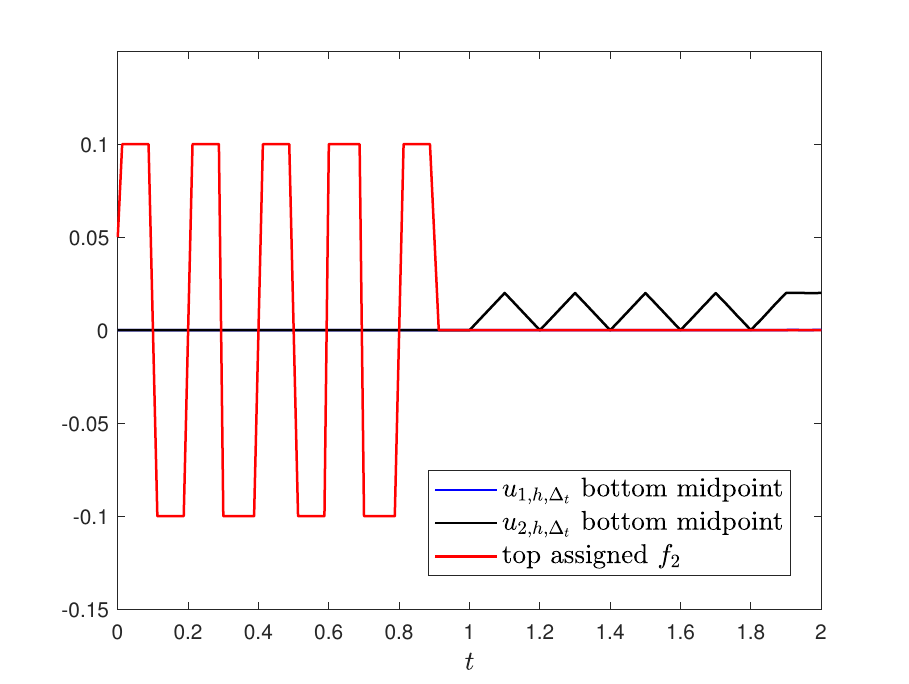}}
\caption{{Displacement as a function of time in the  midpoints of the top (a) and bottom (b) sides, together with the assigned vertical Neumann datum $f_2$.}}
\label{Figure12}
\end{figure}
\begin{figure}[h!]
\centering
\subfloat[]{\includegraphics[scale=0.5]{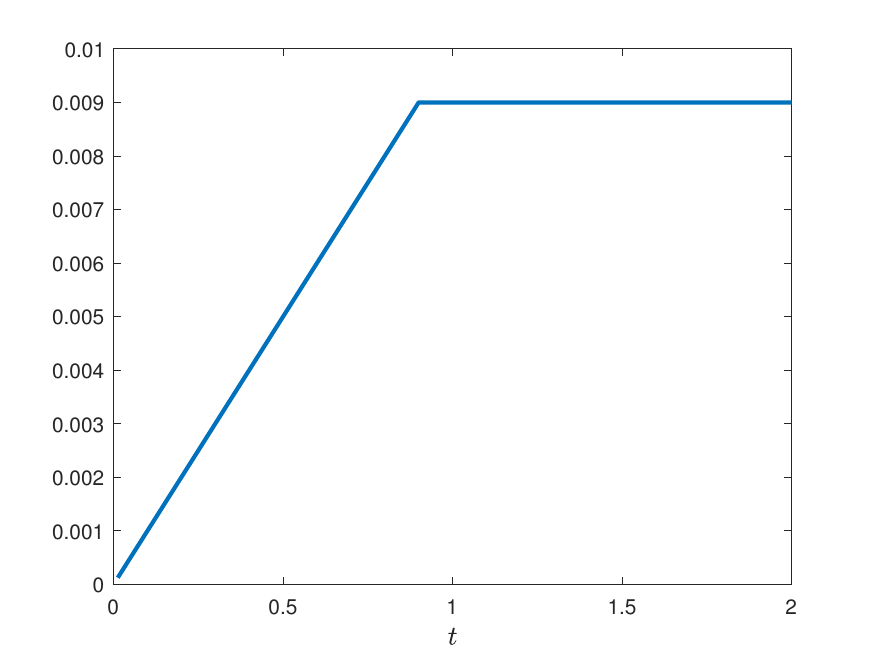}}
\quad
\subfloat[]{\includegraphics[scale=0.5]{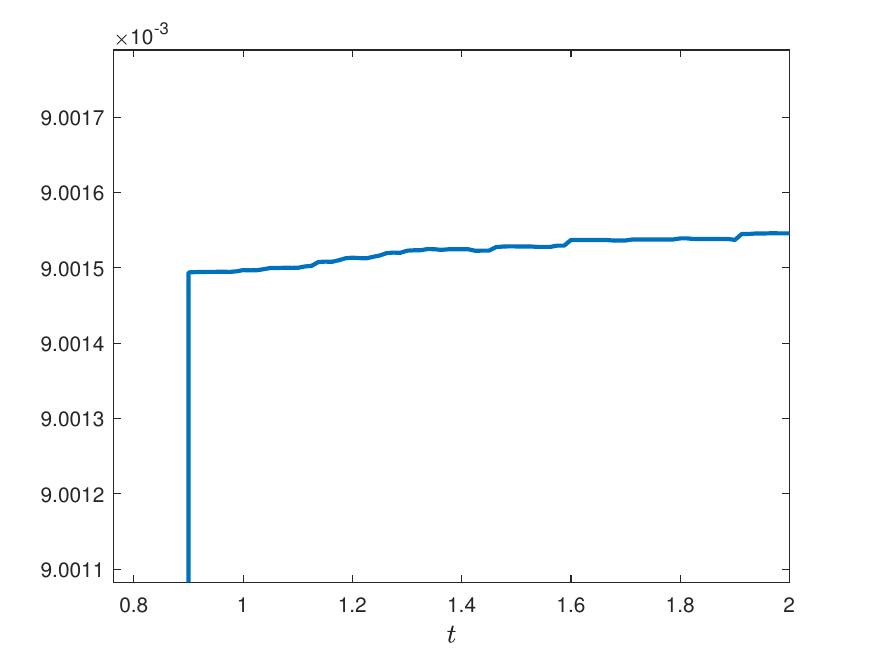}}
\caption{{Energy of the numerical solution as a function of time (a), zoomed-in plot for $t>0.8$ (b).}}
\label{Figure13}
\end{figure}
}

\subsection{Example 4: dynamic contact problem in a circular geometry}

In the last example we take into account the disk $\Omega=\left\lbrace (x,y)^\top \: :\: x^2+y^2\leq 0.2 \right\rbrace$. The contact region is given by the entire boundary of the disk, namely $\Gamma=\Gamma_C$. The fundamental elastodynamic velocities are set as $c_S=1$ and $c_P=2$. We have no Neumann data assigned and the contact force is null, but the dynamic is not trivial since we consider in \eqref{contactbc} the gap function 
$$g(t)=(4\:\sqrt{1-1.5(t-0.5)^2}-4)\: H[1.3-t]-3.2\:H[t-1.3],$$
modeling a contact that push up the disk from the bottom in the time interval $[0.245,0.755]$. A schematic representation of the dynamic contact is reported in Figure \ref{Figure9}.

\begin{figure}
\centering
\includegraphics[scale=0.6]{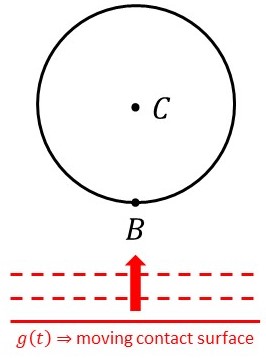}
\caption{Schematic representation of disk and contact boundary: its movement in time is governed by the given gap function $g$.}
\label{Figure9}
\end{figure}

Using the symmetric formulation \eqref{definition_S_symm}, we show in Figure \ref{Figure11} the squared energy error, which decays with the halving of the parameters $h$ and $\Delta t$ (for each level of discretization the ratio $h / \Delta t\simeq 2$ is maintained). The exit test for stopping the Uzawa algorithm consists in fixing $\epsilon=10^{-4}$, and we set the parameter $\rho=10^4$ in the Uzawa update.\\
\begin{figure}[h!]
\centering
\includegraphics[scale=0.6]{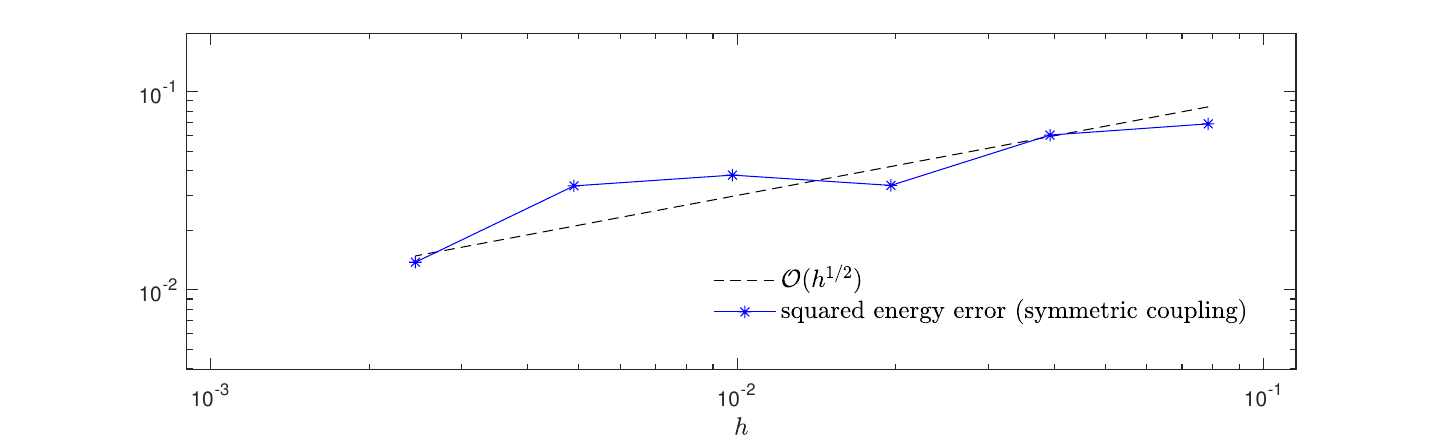}
\caption{Squared error in energy norm for the test on the disk (symmetric formulation).}
\label{Figure11}
\end{figure}
Having fixed $h\simeq 0.02$, Figure \ref{Figure10}(a) shows the time history of the vertical coordinates of the bottom point of $\Omega$, $B$, and its barycenter, $C$, that, before the contact, are identified respectively with the Cartesian coordinates $(0,-0.2)$ and $(0,0)$. It is immediate to observe that during the contact, the bottom point behaves exactly like the gap function $g(t)$, which physically represents a lower bound for the vertical coordinates of the points belonging to $\Omega$, that are dragged up and then released at the end of the contact. Moreover, the approximated displacement at the barycenter, as expected, is characterized by an initial exponential growth and then it becomes linear in time, since after the contact there are no longer external forces applied to the disk $\Omega$. Figure \ref{Figure10}(b) represents instead the global displacement to which the disk is subjected at some time instants: after an initial phase of compression, the points globally move upwards and the deformations are related to the elastic constants that characterize the material of the disk.

\begin{figure}
\centering
\subfloat[]{
\includegraphics[scale=0.6]{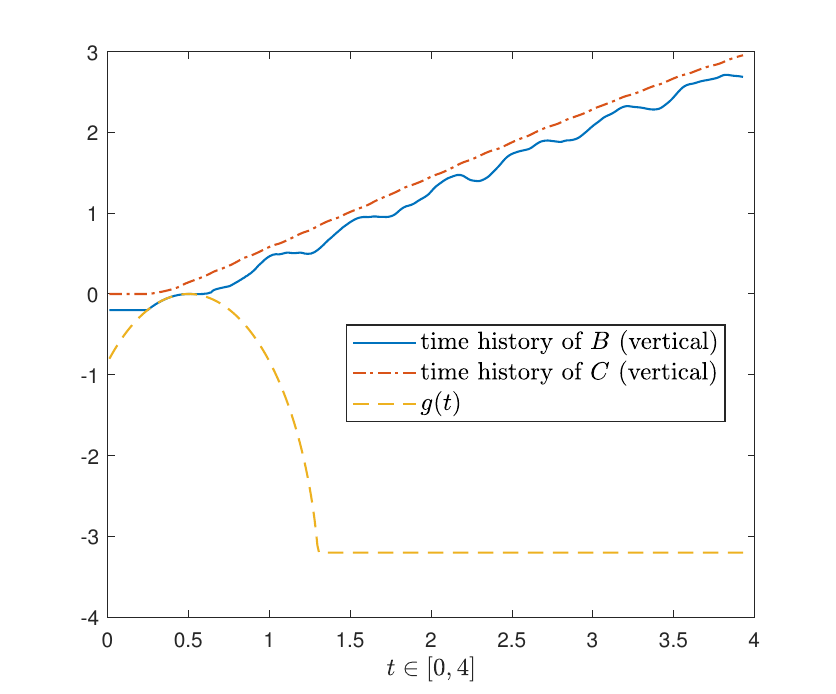}}
\subfloat[]{
\includegraphics[scale=0.6]{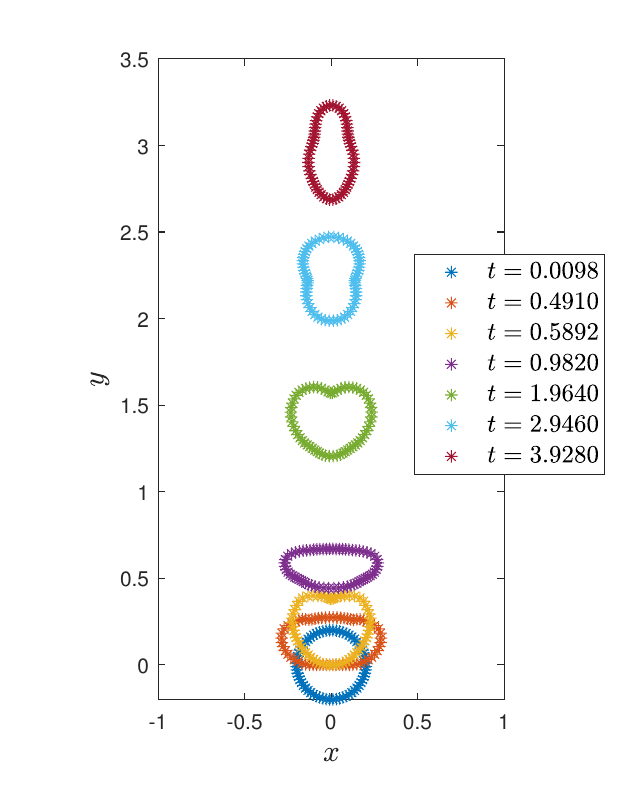}}
\caption{Time history of the bottom and the barycenter points $B$ and $C$ (Figure (a)) and global displacement of the disk $\Omega$ at some time instants (Figure (b)).}
\label{Figure10}
\end{figure}

\section{Conclusions}\label{sec7}
In this work we propose and analyze a Galerkin space-time boundary element method to solve elastodynamic contact problems. Boundary elements provide a natural and efficient formulation, as the contact takes place at the
interface between two bodies.\\
The article presents this approach in the case of a unilateral Signorini contact problem without friction. 
In particular, the algebraic formulation and implementation of an energetic space-time boundary element method are detailed for a mixed formulation of the problem.
Analytically we obtain an a priori error analysis, based on ideas for a scalar variational inequality \cite{contact}. The analysis of the method is crucially
based on an inf–sup condition in space–time. Moreover, the space-time Uzawa iterative algorithm, used to solve the nonlinear problem, is shown to be provably convergent.\\
As a key point, the numerical results indicate stability and convergence for different two-dimensional geometries, including both straight and curved contact boundaries.
 Future work will extend the current approach to two-sided and frictional contact problems, to three-dimensional geometries and space-time adaptive mesh refinements \cite{adaptive}. While the formulation of the proposed approach readily generalizes to Tresca or Coulomb friction \cite{banz, ency, twosided, gwinner}, additional challenges are expected for the nonlinear solver in this case.\\

\section*{Appendix A}\label{appendix A}

This appendix introduces space--time anisotropic Sobolev spaces  as a convenient setting for the analysis of time dependent boundary integral operators. In the case of the wave equation, a detailed exposition may be found in \cite{hd}, and we refer to \cite{ourpaper, Becache1993, Becache1994} for elastodynamics. When $\Gamma$ is an open screen or line segment, so that $\partial\Gamma\neq \emptyset$, we first extend $\Gamma$ to a closed, orientable Lipschitz manifold $\widetilde{\Gamma}$ of dimension $d-1$. 
We recall the definition of Sobolev spaces of supported distributions on $\Gamma$:
$$\widetilde{H}^s(\Gamma) = \{u\in H^s(\widetilde{\Gamma}): \mathrm{supp}\ u \subset {\overline{\Gamma}}\}\ , \quad\ s \in \mathbb{R}\ .$$
The space ${H}^s(\Gamma)$ is defined as the quotient space $ H^s(\widetilde{\Gamma}) / \widetilde{H}^s({\widetilde{\Gamma}\setminus\overline{\Gamma}})$.
We now define a family of Sobolev norms.  Let $\alpha_i$, $i=1, \dots,p$, be a partition of unity subordinate to a covering of $\widetilde{\Gamma}$ by open sets $B_i \subset \mathbb{R}^{d-1}$. If for each $i$, $\varphi_i$ is a diffeomorphism from $B_i$ to the unit cube in $\mathbb{R}^{d-1}$, we define a family of Sobolev norms involving a parameter $\omega \in \mathbb{C}\setminus \{0\}$:
\begin{equation*}
 ||u||_{s,\omega,{\widetilde{\Gamma}}}=\left( \sum_{i=1}^p \int_{{\mathbb{R}^{d-1}}} (|\omega|^2+|\pmb{\xi}|^2)^s|\mathcal{F}\left\{(\alpha_i u)\circ \varphi_i^{-1}\right\}(\pmb{\xi})|^2 d\pmb{\xi} \right)^{\frac{1}{2}}\ .
\end{equation*}
Here, $\mathcal{F}=\mathcal{F}_{\bold{x} \mapsto \pmb{\xi}}$ denotes the Fourier transform $\mathcal{F}\varphi(\pmb{\xi}) = \int e^{-i\bold{x}\cdot\pmb{\xi}} \varphi(\bold{x})\ d\bold{x}$. Different parameters $\omega \in \mathbb{C}\setminus \{0\}$ lead to equivalent norms on $H^s(\Gamma)$, $\|u\|_{s,\omega,\Gamma} = \inf_{v \in \widetilde{H}^s(\widetilde{\Gamma}\setminus\overline{\Gamma})} \ \|u+v\|_{s,\omega,\widetilde{\Gamma}}$ and on $\widetilde{H}^s(\Gamma)$, $\|u\|_{s,\omega,\Gamma, \ast } = \|e_+ u\|_{s,\omega,\widetilde{\Gamma}}$. $e_+$ here denotes the extension by $0$,  which extends a distribution on $\Gamma$ to a distribution on $\widetilde{\Gamma}$.  When a specific, fixed $\omega$ is considered, we write $H^s_\omega(\Gamma)$ for $H^s(\Gamma)$, and $\widetilde{H}^s_\omega(\Gamma)$ for $\widetilde{H}^s(\Gamma)$. 
One observes that $\|u\|_{s,\omega,\Gamma, \ast }\geq \|u\|_{s,\omega,\Gamma}$. 

The space-time anisotropic Sobolev spaces relevant to this article can now be defined as follows:
\begin{definition}
For {$\sigma>0$ and} $r,s \in\mathbb{R}$ we set
\begin{align}
 H^r_\sigma(\mathbb{R}^+,{H}^s(\Gamma))&=\{ u \in \mathcal{D}^{'}_{+}(H^s(\Gamma)): e^{-\sigma t} u \in \mathcal{S}^{'}_{+}(H^s(\Gamma))  \textrm{ and }   ||u||_{r,s,\Gamma} < \infty \}\ , \nonumber \\
 H^r_\sigma(\mathbb{R}^+,\widetilde{H}^s({\Gamma}))&=\{ u \in \mathcal{D}^{'}_{+}(\widetilde{H}^s({\Gamma})): e^{-\sigma t} u \in \mathcal{S}^{'}_{+}(\widetilde{H}^s({\Gamma}))  \textrm{ and }   ||u||_{r,s,\Gamma, \ast} < \infty \}\ .\label{sobdef}
\end{align}
Here, $\mathcal{D}^{'}_{+}({H}^s({\Gamma}))$ denotes the space of all distributions on $\mathbb{R}$ with support in $[0,\infty)$, with values in the real-valued subspace of the Hilbert space ${H}^s({\Gamma})$, and $\mathcal{D}^{'}_{+}(\widetilde{H}^s({\Gamma}))$ is defined in an analogous way. $\mathcal{S}^{'}_{+}({H}^s({\Gamma}))\subset \mathcal{D}^{'}_{+}({H}^s({\Gamma}))$ and $\mathcal{S}^{'}_{+}(\widetilde{H}^s({\Gamma}))\subset \mathcal{D}^{'}_{+}(\widetilde{H}^s({\Gamma}))$ denote the subspaces of tempered distributions.  The Sobolev spaces are equipped with the norms
\begin{align}
\|u\|_{r,s,\sigma}:=\|u\|_{r,s,\Gamma,\sigma}&=\left(\int_{-\infty+i\sigma}^{+\infty+i\sigma}|\omega|^{2r}\ \|\hat{u}(\omega)\|^2_{s,\omega,\Gamma}\ d\omega \right)^{\frac{1}{2}}\ ,\nonumber \\
\|u\|_{r,s,\sigma,\ast} := \|u\|_{r,s,\Gamma,\sigma,\ast}&=\left(\int_{-\infty+i\sigma}^{+\infty+i\sigma}|\omega|^{2r}\ \|\hat{u}(\omega)\|^2_{s,\omega,\Gamma,\ast}\ d\omega \right)^{\frac{1}{2}}\,. \label{sobnormdef}  
\end{align}
\end{definition}
They are Hilbert spaces. For $r=s=0$ they correspond to the weighted $L^2$-space with the scalar product $\langle u,v \rangle_{{\sigma,\Gamma,\mathbb{R}^+}}=\int_0^\infty e^{-2\sigma t} \int_\Gamma u \,v \, d\Gamma_{\bold{x}}\ dt$ {(see \eqref{sigma_product} in Subsection \ref{sec:Boundaryintegralformulation}).} 

The boundary integral operators for the elastodynamic problem obey the following mapping properties between these spaces:
\begin{theorem}\label{mappingproperties}
Let 
$\sigma>0$. Then following operators are continuous for $r\in \R$:
\begin{align*}
& \mathcal{V}:  {H}^{r+1}_\sigma(\R^+, \tilde{H}^{-\frac{1}{2}}(\Gamma))^d\to  {H}^{r}_\sigma(\R^+, {H}^{\frac{1}{2}}(\Gamma))^d \ ,
\\ & \mathcal{K}^\ast:  {H}^{r+1}_\sigma(\R^+, \tilde{H}^{-\frac{1}{2}}(\Gamma))^d\to {H}^{r}_\sigma(\R^+, {H}^{-\frac{1}{2}}(\Gamma))^d \ ,
\\ & \mathcal{K}:  {H}^{r+1}_\sigma(\R^+, \tilde{H}^{\frac{1}{2}}(\Gamma))^d\to {H}^{r}_\sigma(\R^+, {H}^{\frac{1}{2}}(\Gamma))^d \ ,
\\ & \mathcal{W}:  {H}^{r+1}_\sigma(\R^+, \tilde{H}^{\frac{1}{2}}(\Gamma))^d\to {H}^{r}_\sigma(\R^+, {H}^{-\frac{1}{2}}(\Gamma))^d \ .
\end{align*}
\end{theorem}

\vspace{0.2in}
\noindent \textbf{Acknowledgements}: This work has been partially supported by the University of Parma with the project Fil2020 - Action A1 “Time-domain Energetic BEM for elastodynamic problems, with advanced applications”. Early stages of this research were supported through the ``Oberwolfach Research Fellows'' program in 2020.

\end{document}